\documentclass{amsart}
\usepackage[utf8]{inputenc}
\usepackage{amsmath}
\usepackage{amsthm}
\usepackage{amsfonts, dsfont}
\usepackage{amssymb}
\usepackage[all]{xy}
\usepackage{indentfirst}
\usepackage[pagebackref=true]{hyperref}

\usepackage{cleveref}
\usepackage[alphabetic, initials]{amsrefs}

\usepackage{xfrac}

\newtheorem{thm}{Theorem}[section]

\newtheorem{theorem}[thm]{Theorem}
\newtheorem{corollary}[thm]{Corollary}
\newtheorem{proposition}[thm]{Proposition}

\newtheorem{lemma}[thm]{Lemma}
\newtheorem*{theorem*}{Theorem}
\newtheorem*{conjecture*}{Conjecture}
\newtheorem*{corollary*}{Corollary}

\newtheorem{thmintro}{Theorem}

\theoremstyle{definition}

\newtheorem*{defn*}{Definiton}
\newtheorem{example}[thm]{Example}
\newtheorem{definition}[thm]{Definition}
\newtheorem*{ack}{Acknowledgements}
\newtheorem{remark}[thm]{Remark}

\newcommand{\N}{\mathbb{N}} %% Naturals
\newcommand{\Z}{\mathbb{Z}} %% Integers
\newcommand{\C}{\mathbb{C}} %% Complex
\newcommand{\R}{\mathbb{R}}

\usepackage{verbatim} %% Enables 'comment'
\usepackage{color}
\newcommand{\G}{\Gamma}

\newcommand{\cs}{\mathrm{C}^*}

\DeclareMathOperator{\interior}{int}

\newcommand{\la}{\lambda}

\newcommand{\ess}{\mathrm{ess}}

\newcommand{\act}{\!\curvearrowright\!}

\newcommand{\Fix}{\operatorname{Fix}}

\newcommand{\supp}{\operatorname{supp}}

\newcommand{\dom}{\operatorname{dom}}
\newcommand{\Ped}{\mathrm{Ped}}

\newcommand{\cL}{\mathcal{L}}

\newcommand{\Iso}{\mathrm{Iso}(G)}

\newcommand{\Rad}{\mathrm{Rad}}

\newcommand{\range}{\mathbf{r}}
\newcommand{\source}{\mathbf{s}}
\newcommand{\cCc}{\mathcal{C}_c}

\title{Invariant measures and traces on groupoid $\boldsymbol{\mathrm{C}^\ast}$-algebras}

﻿
\author[A. Miller]{Alistair Miller}
\address[A.~Miller]{%
    Department of Mathematics\\
    KU Leuven\\
    Celestijnenlaan 200B\\
 box 2400\\
    3001 Leuven\\
    Belgium}
\email{alistair.miller@kuleuven.be}

\author[E. Scarparo]{Eduardo Scarparo}
\address[E. Scarparo]{Center for Engineering, Federal University of Pelotas, R. Benjamin Constant 989,
	96010-020, Pelotas/RS, Brazil}
\email{eduardo.scarparo@ufpel.edu.br}

\begin{document}

\subjclass[2020]{46L05}
    
    	\begin{abstract}
We provide sufficient conditions for the existence of a trace on the essential $\mathrm{C}^\ast$-algebra of a (not necessarily Hausdorff) \'etale groupoid $G$ which extends an invariant measure $\mu$ on the unit space of $G$. In particular, it suffices for the isotropy groups of $G$ to be amenable, or for $G$ to be essentially free with respect to $\mu$. 

We also show that $G$ is essentially free with respect to an invariant measure $\mu$ if and only if $\mu$ extends to a unique trace on the full $\mathrm{C}^\ast$-algebra of $G$.

We work in the generality of possibly infinite measures and, accordingly, possibly unbounded traces. Moreover, whenever possible, we state our results for twisted groupoids.

As an application, we show that gauge-invariant algebras of finite-state self-similar groups admit a unique tracial state.
	\end{abstract}
	
	\maketitle

    \section{Introduction}

The construction of a $\cs$-algebra from a twisted \'etale groupoid provides a framework for $\cs$-algebras built from dynamical, combinatorial and group-theoretic data. On top of unifying many natural constructions of $\cs$-algebras, the $\cs$-algebras of twisted \'etale groupoids exhaust the class of Elliott-classifiable $\cs$-algebras \cite{Li20}. The Classification Programme \cite{Win18, GLN23} has classified these $\cs$-algebras by the Elliott invariant, which consists of K-theoretic and tracial data. The focus of this article is on understanding and computing the trace spaces of twisted \'etale groupoid $\cs$-algebras.

Examples of \'etale groupoids coming from dynamics and geometry are often non-Hausdorff \cite[Chapter 5]{MM03}, \cite{CEPSS}, \cite[Example 7.8]{KS22a}. The groupoid of germs of a group action locally encodes the homeomorphisms of the action. If the action is not topologically free, this groupoid can be non-Hausdorff, such as for the action of the Grigorchuk group on the boundary of the $2$-regular rooted tree. Moreover, the theory of self-similar groups provides an array of (often simple, nuclear) $\cs$-algebras with non-Hausdorff groupoid models (see \cite{Nek09,MS24} for some K-theory computations). Twists over Hausdorff \'etale groupoids arise through Kumjian--Renault reconstruction \cite{Kum86, Ren08} from the data of a Cartan pair. This has recently been extended to the non-Hausdorff setting, using the weaker notion of an essential Cartan pair \cite{ExelPitts22,Tay25}.

Non-Hausdorffness of an \'etale groupoid $G$ introduces a series of technical challenges in the study of its $\cs$-algebras. Functions in the convolution algebra $\cCc(G)$ need not be continuous, so we do not have access to a genuine conditional expectation $E \colon \cs(G) \to C_0(G^0)$. Moreover, the ideal structure of the reduced $\cs$-algebra $\cs_r(G)$ may become poorly behaved, which is remedied through the introduction of the essential $\cs$-algebra $\cs_\ess(G)$ \cite{KM21}, a quotient of $\cs_r(G)$. It is the essential $\cs$-algebra which is used to perform Kumjian--Renault reconstruction in the non-Hausdorff setting. 

By a \emph{trace} on a $\cs$-algebra $A$, we mean a positive tracial linear functional on the Pedersen ideal $\Ped(A)$. For unital $A$ we have $\Ped(A) = A$ and every trace is bounded. Any trace on the full $\cs$-algebra $\cs(G,\mathcal L)$ of a twisted \'etale groupoid $(G,\mathcal L)$ restricts to an invariant Radon measure on $G^0$, and conversely every invariant Radon measure on $G^0$ has a canonical extension to a trace on $\cs_r(G,\mathcal L)$ (thus also for $\cs(G,\mathcal L)$).

However, this breaks down for the essential $\cs$-algebra, because the canonical trace defined on the reduced $\cs$-algebra for a given invariant Radon measure may not vanish on the ideal determining the essential $\cs$-algebra. To address this issue, Kwaśniewski, Meyer and Prasad assume in some of the results in \cite{KMP} that the reduced and essential $\cs$-algebras coincide. 

Nonetheless, the measure may extend to some other trace on $\cs_\ess(G, \mathcal L)$; in fact we do not know whether every invariant Radon measure admits an extension to a trace on the essential $\cs$-algebra. 
In Section \ref{section:ess} we obtain the following sufficient criteria for the existence of an extension:

\begin{thmintro}\label{tie}
Let $(G,\mathcal L)$ be a twisted étale groupoid and $\mu$ an invariant Radon measure on $G^0$. Suppose that at least one of the following conditions holds:
\begin{enumerate}
\item Every isotropy group of $G$ is amenable and the twist $\mathcal L$ is trivial;
\item $G$ is essentially free with respect to $\mu$;\label{ess free intro}
\item $G$ is a group bundle and $\mu$ is a probability measure.
\end{enumerate} 
Then there exists a trace on $\cs_\ess(G,\mathcal L)$ which extends $\mu$. In case \eqref{ess free intro}, the canonical extension of $\mu$ descends to $\cs_\ess(G,\mathcal L)$.
\end{thmintro}
In particular, if $G$ is amenable, then every invariant Radon measure on $G^0$ admits an extension to a trace on $\cs_\ess(G)$.

Care has to be taken with the definition of essential freeness in the above theorem in full generality (see Definition \ref{def:ess free}), but under mild conditions,\footnote{It suffices for example for $G$ to be second countable.} $G$ is essentially free with respect to $\mu$ if $\mu(\{x \in G^0 \mid G^x_x \ne \{x\}\}) = 0$.  
If an \'etale groupoid $G$ is essentially free with respect to an invariant Radon measure with full support, then it is topologically free. On the other hand, if $G$ satisfies one of various notions of almost finiteness \cite{Mat12,Ker20}, it will be essentially free with respect to any invariant Radon measure (see \cite[Remark 6.6]{Mat12} and \cite[Lemma 2.7]{GGGKN24}).

In \cite[Theorem 2.7]{KTT}, Kawamura, Takemoto and Tomiyama showed that essential freeness of the action of a group $\G$ on a compact space $X$ with respect to an invariant Radon probability measure $\mu$ is equivalent to $\mu$ admitting a unique extension to a tracial state on the crossed product $C(X)\rtimes\G$. This result was extended to the setting of Hausdorff \'etale groupoids in \cite{LZ26}, and can be deduced for second countable non-Hausdorff \'etale groupoids from \cite[Corollary 2.4]{NS22}. We push the result further to include infinite measures, and show that only one direction holds in the twisted setting.  Whereas the proof in \cite{NS22} relies on Renault's disintegration theorem, our approach in this work is more direct, in the spirit of \cite{KTT} and \cite{LZ26}. 

\begin{thmintro}[Theorems \ref{thm:u} and \ref{thm:nu}]\label{tif}
Let $\mu$ be an invariant Radon measure on the unit space of an étale groupoid $G$ and let $\mathcal L$ be a twist over $G$. Then $\mu$ extends to a unique trace on $\cs(G,\mathcal L)$ if $G$ is essentially free with respect to $\mu$. The converse holds for the trivial twist.
\end{thmintro}

For a groupoid $G$ which is essentially free with respect to every invariant Radon measure on $G^0$, this puts traces on $\cs_\ess(G,\mathcal L)$ in bijection with invariant Radon measures. We show that this bijection is moreover a homeomorphism.

\begin{thmintro}[Corollary \ref{cor:trace space}]
Let $(G,\mathcal L)$ be a twisted \'etale groupoid such that $G$ is essentially free with respect to every invariant Radon measure. Then the space of traces on $\cs_\ess(G,\mathcal L)$ is isomorphic to the space of invariant Radon measures on $G^0$.
\end{thmintro}

In Section \ref{sec:ssg}, we apply Theorems \ref{tie} and \ref{tif} to gauge-invariant algebras of finite-state self-similar groups to see that they admit a unique tracial state.

\begin{ack}
The first author was supported by the Independent Research Fund Denmark through Grant 1054-00094B and by the Research Foundation Flanders (FWO) through Project 1212126N.

The authors thank Kang Li and Ben Steinberg for useful discussions, and the first author thanks Grigoris Kopsacheilis for conversations on essential freeness. We also thank the anonymous referee for their helpful comments and for calling our attention to \cite{EP16}.
\end{ack}

\section{Preliminaries}

An \emph{étale groupoid} $G$ is a topological groupoid whose unit space $G^0$ is a locally compact Hausdorff space and such that the range and source maps $\range,\source\colon G\to G^0$ are local homeomorphisms. By \cite[Proposition 3.2]{Ex08}, this implies that $G^0$ is an open subset of $G$. Moreover, $G$ is Hausdorff if and only if $G^0$ is also a closed subset of $G$ (\cite[Lemma 8.3.2]{SSW20}). The \emph{isotropy} of $G$ is the set $\mathrm{Iso}(G):=\{g\in G \mid \range(g)=\source(g)\}$, which is closed in $G$. A \emph{bisection} is an open subset $U\subset G$ such that the restrictions of $\range$ and $\source$ to $U$ are homeomorphisms onto their images. 

Given $x\in G^0$, let $G_x:=\source^{-1}(x)$ and $G^x:=\range^{-1}(x)$. The \emph{isotropy group} at $x$ is $G_x^x:=\source^{-1}(x)\cap \range^{-1}(x)$. Given $H\leq G_x^x$, we let $G_x/H:=\{gH \mid g\in G_x\}$. 

\subsection*{Groupoids of germs}
A \emph{pseudogroup} is a pair $(S,X)$ where $X$ is a locally compact Hausdorff space and $S$ is a set of homeomorphisms between open subsets of $X$, which is closed under composition and inversion and such that the family of domains $(\dom s)_{s\in S}$ covers $X$.

In what follows, we recall the construction of the groupoid of germs associated to $(S,X)$ (see \cite[Section 3]{Ren08}). 
Let $\Omega:=\{(s,x)\in S\times X \mid x\in \dom s\}$. Given $(s,x)$ and $(t,y)$ in $\Omega$ we will say that $(s,x)\sim(t,y)$ if $x=y$ and there is a neighbourhood $U$ of $x$ such that $U\subset\dom s\cap \dom t$ and $s|_U=t|_U$. The equivalence class of $(s,x)$ will be denoted by $[s,x]$. The quotient $G(S,X): =\frac{\Omega}{\sim}$ is called the \emph{groupoid of germs} of $(S,X)$.
The product of two elements $[t,y]$ and $[s,x]$ is defined if and only if $y=sx$, in which case $[t,y][s,x]=[ts,x]$. Inversion is given by $[s,x]^{-1}:=[s^{-1},sx]$.

Given $s\in S$ and $A\subset \dom s$, let $\Theta( s,A):=\{[s,x] \mid x\in A\}$ and set $\Theta_s := \Theta(s,\dom s)$. The topology on $ G(S,X)$ is the one generated by the basis $\{\Theta( s,U) \mid s\in S, \, \text{$U\subset \dom s$ open}\}$. With this topology, $ G(S,X) $ is an étale groupoid. Moreover, for each $s\in S$ and open set $U\subset \dom s$, the set $\Theta( s,U)$ is a bisection. We identify $G(S,X)^0$ with $X$. For $s \in S$, we consider the set $\Fix_s = \{x \in X \mid x \in \dom s, \, sx = x \}$ of points fixed by $s$. Unpacking the definitions gives the following equalities (see also \cite[Proposition 3.14]{EP16}). 
\begin{lemma}\label{eq:fixb}
Let $(S,X)$ be a pseudogroup. Given $s \in S$, we have $\Theta_s \cap X = \Theta({s,\interior\Fix_s})$ and $\Theta_s \cap \mathrm{Iso}(G(S,X)) = \Theta({s,\Fix_s})$. Consequently, 
\begin{equation*}
\Theta_s\cap \mathrm{Iso}(G(S,X))\setminus X= \Theta(s,\Fix_s \setminus\interior\Fix_s)
\end{equation*}
and $\source(\Theta_s \cap \mathrm{Iso}(G(S,X)) \setminus X) = \Fix_s \setminus\interior\Fix_s$.
\end{lemma}

The following characterization of the Hausdorff property for groupoids of germs will not be used in the rest of the paper, but it might be of independent interest (see also \cite[Theorem 3.15]{EP16} and \cite[Lemma 4.3]{NO19}).
\begin{proposition}\label{prop:Hd}
Let $(S,X)$ be a pseudogroup. Then $G(S,X)$ is Hausdorff if and only if, for each $s\in S$, the set $\interior\Fix_s$ is closed in $\dom s$.
\end{proposition}
\begin{proof}
$(\implies)$: Suppose that there is $s\in S$ such that $\interior\Fix_s$ is not closed in $\dom s$. Take a net $(x_i)\subset  \interior\Fix_s$ with $x_i\to x\in \Fix_s\setminus \interior\Fix_s$. Then $([s,x_i])_i\subset  X$ and $[s,x_i]\to [s,x]\notin X$, thus showing that $X$ is not closed in $G(S,X)$.

$(\impliedby)$: Suppose that $X$ is not closed in $G(S,X)$. Take a net $([s_i,x_i])\subset  X$ with $[s_i,x_i]\to [s,x]\notin X$. Take $i_0$ such that $i\geq i_0$ implies $[s_i,x_i]\in\Theta(s,\dom s)$. In particular, for $i\geq i_0$, we get $x_i\in\interior \Fix_s$. Since $x_i\to x$ and $x\notin \interior \Fix_s$, we conclude that $\interior\Fix_s$ is not closed in $\dom s$.
\end{proof}
\begin{remark}
In \cite[Section 4]{Ex08}, Exel associated to any action of an inverse semigroup on a locally compact Hausdorff space a groupoid that he also called the groupoid of germs of the action. A pseudogroup $(S,X)$ may be viewed as a faithful action of $S$ on $X$. However, in general, in order for the notion of groupoid of germs $G(S,X)$ that we work with to coincide with Exel's construction, it is necessary to assume that the domains of $S$ form a basis for $X$. This can always be arranged by enlarging $S$ so that it is closed under restriction to open subsets of $X$, which of course does not change $G(S,X)$. Furthermore, if the domains of $S$ form a basis for $X$, then Proposition \ref{prop:Hd} is a special case of \cite[Theorem 3.15]{EP16}.
\end{remark}

\subsection*{Twisted groupoids}
In this subsection, we recall the picture of twisted étale groupoids using Fell line bundles (see \cite[Remark 7.7]{KM21}).

\begin{definition}\label{def:flb}
Let $G$ be an étale groupoid. A \emph{Fell line bundle} over $G$ consists of the following data: a topological space $\mathcal L$, a continuous open surjection $p \colon \mathcal L \to G$ and the structure of a one-dimensional complex normed space on each fibre $\mathcal L_g := p^{-1}(g)$. It is moreover equipped with bilinear multiplication maps $\cdot \colon \mathcal L_g \times \mathcal L_h \to \mathcal L_{gh}$ (when $\source(g) = \range(h)$) and conjugate-linear adjoint maps $\ast \colon \mathcal L_g \to \mathcal L_{g^{-1}}$, such that:
\begin{enumerate}
\item Addition $+\colon \{ (a,b) \in \mathcal L \times \mathcal L \mid p(a) = p(b) \} \to \mathcal L$ is continuous.
\item Scalar multiplication $\mathbb C \times \mathcal L \to \mathcal L$ is continuous.
\item The norm $\| \cdot \| \colon \mathcal L \to \mathbb R$ is continuous.
\item If $g \in G$ and $(a_i)$ is a net in $\mathcal L$ with $p(a_i) \to g$ and $\|a_i\| \to 0$, then $a_i \to 0_g$ in $\mathcal L$. 
\item Multiplication $\cdot \colon \{(a,b) \in \mathcal L \times \mathcal L \mid \source(p(a)) = \range(p(b)) \} \to \mathcal L$ is continuous, associative and makes the norm multiplicative. 
\item The adjoint $\ast \colon \mathcal L \to \mathcal L$ is continuous, involutive, antimultiplicative and norm-preserving.
\item For each $a \in \mathcal L$, we have $a^*a\geq 0$.\label{positivity}
\end{enumerate}

We call the pair $(G,\cL)$ a \emph{twisted étale groupoid}.
\end{definition}
Note that, for each $a\in \cL$, we have $\|a^*a\| = \|a\|^2$. This gives the fibre $\mathcal L_x$ at each unit $x \in G^0$ the structure of a $\cs$-algebra, which must be $\mathbb C$. Positivity of $a^*a$ in %Definition \ref{def:flb} 
\eqref{positivity} is to be interpreted with respect to this C*-algebra. 

Given $g\in G$, it is not difficult to see that the restricted topology on $\mathcal L_g$ coincides with the norm topology.

Given $U\subset G$, let $\cL|_U:=p^{-1}(U)$. If $U$ is Hausdorff, then $\cL|_U$ is Hausdorff as well. Indeed, take $v,w\in \cL|_U$ distinct points. If $p(v)\neq p(w)$, then it follows from continuity of $p$ that $v$ and $w$ can be separated by open sets. Suppose $p(v)=p(w)$ and that there is a net $(z_i)$ in $\cL|_U$ such that $z_i\to v,w$. Continuity of addition, scalar multiplication and the norm implies that $\| 0_{p(z_i)}\| \to \|v-w\|$, and thus $v=w$, showing the claim. In particular, $\mathcal L|_U\to U$ is a Banach bundle in the sense of \cite[II.13.4 Definition]{FellDoranI} and, if $U$ is also open, by \cite[C17 Theorem]{FellDoranI}, given $b\in \mathcal L|_U$, there is a continuous section $f\colon U\to \mathcal L$ such that $f(p(b))=b$.

\begin{example}
The \emph{trivial Fell line bundle} over $G$ is $\cL=G\times \C$, with the product topology, and multiplication and adjoint maps inherited from $\C$.
\end{example}

For a unit $x \in G^0$, we write $1_x$ for the unit of the one-dimensional $\cs$-algebra $\mathcal L_x$. Note that, given $a\in\cL|_{G_x}$ and $b\in \cL|_{G^x}$, we have $a1_x = a$ and $  1_x b= b$.

We say that an open and Hausdorff subset $U\subset G$ is \emph{trivialisable} (with respect to $\cL$) if there exists an isometric isomorphism (in the sense of \cite[II.13.8 Definition]{FellDoranI}) between the trivial Banach bundle $U\times \C$ and $\cL|_U$. 

\begin{proposition}\label{prop:trivial}
Let $(G,\mathcal L)$ be a twisted \'etale groupoid.
\begin{enumerate}
\item The groupoid $G$ can be covered by trivialisable sets.

\item The map $(x,\lambda) \mapsto \lambda 1_x \colon G^0 \times \mathbb C \to \mathcal L|_{G^0}$ is an isometric isomorphism which preserves multiplication and the adjoints.
\end{enumerate}
\end{proposition}
\begin{proof}
(1) Fix $g_0\in G$ and take an open Hausdorff neighbourhood $U$ of $g_0$ for which there exists a continuous section $f\colon U\to\cL$ with $f(g) \ne 0_g$ for each $g\in U$. By normalising, we can assume that $\|f(g)\|=1$ for each $g\in U$. 

Let $T\colon U\times\C\to \cL|_U$ be given by $T(g,\lambda):=\lambda f(g)$. We claim that $T$ is an isometric isomorphism. The only subtle part is showing that, given a net $(g_i,\lambda_i)\subset U\times \C$ with $\lambda_if(g_i)\to T(g,\lambda)=\lambda f(g)\in \cL|_U$, we get $\lambda_i\to \lambda$. Indeed, we have that $\lambda_if(g_i)-\lambda f(g_i)\to 0_{g},$ hence $|\lambda_i-\lambda|\to 0$.

(2) By the above argument, it suffices to show that the section $x \mapsto 1_x$ is continuous. Fix $x\in G^0$ and take a continuous section $\eta \colon G^0 \to \mathcal L$ with $\eta(x) = 1_x$ (using \cite[C17 Theorem]{FellDoranI}). Let $(x_i)$ be a net with $x_i \to x$ and let $\lambda_i \in \mathbb C$ satisfy $\eta(x_i) = \lambda_i 1_{x_i}$. Then $\eta(x_i) \to 1_x$, so by continuity of the norm, multiplication and addition we have $|\lambda_i| \to 1$ and $|\lambda_i^2-\lambda_i| \to 0$. It follows that $\lambda_i \to 1$, and therefore $1_{x_i} \to 1_x$. 
\end{proof}

 \subsection*{Groupoid $\boldsymbol{\cs}$-algebras}
Let $(G,\cL)$ be a twisted étale groupoid. For an open Hausdorff set $U \subset  G$, we write $C_c(U, \mathcal L)$ for the space of compactly supported continuous sections $U \to \mathcal L$. We view an element of $C_c(U, \mathcal L)$ as a section on $G$ by setting the value to be $0$ outside $U$. 

Let $\cCc(G, \mathcal L)$ be the linear span with the space of sections on $G$ of the union of $C_c(U, \mathcal L)$ for all open and Hausdorff sets $U$. If $ G $ is not Hausdorff, then the sections in $\cCc( G,\cL )$ are not necessarily continuous, but they are Borel measurable. 
\begin{example}
If $\cL=G\times\C$ is the trivial Fell line bundle, we can see elements of $\cCc(G,\cL)$ as complex-valued functions. In this case, we shall usually omit $\cL$, thus writing $\cCc(G)$ instead of $\cCc(G,\cL)$, for instance.
\end{example}

The proof of the following result is analogous to \cite[Proposition 3.10]{Ex08}.

\begin{proposition}\label{prop:covera}
Let $(G,\mathcal L)$ be a twisted \'etale groupoid. If $\mathcal U$ is an open cover of $G$ by Hausdorff sets, then $\cCc(G,\mathcal L) = \sum_{U \in \mathcal U} C_c(U,\mathcal L)$.

\end{proposition}

The vector space $\cCc(G, \mathcal L)$ has the structure of a $*$-algebra with product given by $(\xi_1 \xi_2)(g):=\sum_{g_1g_2=g}\xi_1(g_1)\xi_2(g_2)$ and involution by $\xi_1^*(g)=\xi_1(g^{-1})^*$, for $\xi_1,\xi_2\in \cCc(G, \mathcal L)$ and $g\in G$. Via Proposition \ref{prop:trivial}, we identify the $*$-algebras $C_c(G^0,\cL)$ and $C_c(G^0)$ and, more generally, sections $G^0 \to \cL$ are identified with functions $G^0 \to \mathbb C$.

Given $*$-representations $\pi$ and $\sigma$ of $\cCc(G,\cL)$, we say that \emph{$\sigma$ weakly contains $\pi$}, written $\pi\prec\sigma$, if $\|\pi(f)\|\leq \|\sigma(f)\|$ for each $f\in \cCc(G,\cL)$. We denote by $\cs_{\pi}(G,\cL)$ the C*-algebra generated by $\pi(\cCc(G,\cL))$.

The \emph{full $\cs$-algebra} $\cs( G ,\cL)$ of $ (G,\cL) $ is the completion of $\cCc( G,\cL )$ under the $\cs$-norm given by $\sup\{\|a\| \mid \text{$\|\cdot\|$ is a $\cs$-seminorm on $\cCc(G,\cL)$}\},$ for $a\in \cCc( G,\cL )$ (by an argument similar to \cite[Proposition 3.14]{Ex08}, this norm is bounded). The $\cs$-identity implies that, for any bisection $U \subset G$ and any $a \in C_c(U,\mathcal L)$, the norm is given by $\|a\| = \sup_{g \in U}\|a(g)\|$.
Since for any $\xi \in \cCc(G,\mathcal L)$, there is $f \in C_c(G^0)$ with $f \xi = \xi f = \xi$, then any approximate unit for $C_c(G^0)$ is an approximate unit for $\cs(G,\mathcal L)$. 
 
For each $x \in G^0$, the \emph{left regular representation} $\lambda_x$ of $(G,\mathcal L)$ is defined on the Hilbert space $\ell^2(G_x,\mathcal L)$ of $\ell^2$-sections $G_x \to \mathcal L$, with $\lambda_x \colon \cCc(G,\mathcal L) \to \mathcal B(\ell^2(G_x,\mathcal L))$ given by $\lambda_x(a)(\xi) \colon g \mapsto \sum_{g_1 g_2 = g} a(g_1) \xi(g_2)$ for $a \in \cCc(G,\mathcal L)$, $\xi \in \ell^2(G_x,\mathcal L)$ and $g \in G_x$. The \emph{reduced $\cs$-algebra} $\cs_r(G,\mathcal L)$ of $(G,\cL)$ is the completion of $\cCc(G,\mathcal L)$ with respect to the norm $\|a\|_r := \sup_{x \in G^0} \|\lambda_x(a)\|$.

Given a locally compact Hausdorff space $X$, let $B_\infty(X)$ be the $\cs$-algebra of bounded Borel measurable functions on $X$, endowed with the sup norm.
Given $x\in G^0$ and $a\in \cCc(G,\cL)$, note that $\langle\lambda_{x}(a)\delta_x,\delta_x\rangle=a(x)$. Furthermore, the function $a|_{G^0}$ is Borel measurable. Therefore, there is a contractive positive linear map $E\colon \cs(G,\cL)\to B_{\infty}(G^0)$ such that $E(a)=a|_{G^0}$ for all $a\in \cCc(G,\cL)$.

Given a locally compact Hausdorff space $X$, let $M(X)$ be the ideal of $B_{\infty}(X)$ consisting of functions which vanish off a meagre set. 
Let $J:=\{a\in \cs(G,\mathcal L) \mid E(a^*a)\in M(G^0)\}$. Then $J$ is an ideal of $\cs(G,\mathcal L)$ (see \cite[Chapter 3]{ExelPitts22} and \cite[Section 4]{BKM}). 

The \emph{essential $\cs$-algebra} of $(G,\cL)$ is $\cs_{\ess}(G,\mathcal L):=\cs(G,\mathcal L)/J$. Given $a\in \cCc(G,\mathcal L)$, we let $[a]:=a+J\in \cs_{\ess}(G,\mathcal L)$. Note that $C_c(G^0)\cap J=\{0\}$. We also note that given an open subgroupoid $H \subset G$ which contains $G^0$, the inclusion $\cCc(H,\mathcal L|_H) \hookrightarrow \cCc(G,\mathcal L)$ descends to an inclusion $\cs_\ess(H,\mathcal L|_H) \hookrightarrow \cs_\ess(G,\mathcal L)$. 

\begin{remark}\label{countable remark}
It is sometimes convenient to assume that an \'etale groupoid $G$ can be covered by countably many bisections. The following idea outlined in \cite{Hum25} can often be used to reduce to this case. Given a countable set $\mathcal U$ of bisections, note that the groupoid algebraically generated by $\mathcal U$ is open and can be covered by countably many bisections. Hence $G$ can be written as a directed union $G=\bigcup_k H_k$ of open subgroupoids which contain $G^0$ and can be covered by countably many bisections. For any twist $\mathcal L$ over $G$ we can even arrange that $\cs_\ess(G,\mathcal L)= \bigcup_k \cs_\ess(H_k, \mathcal L|_{H_k})$, since for any $a \in \cs_\ess(G,\mathcal L)$ we may pick $a_n \in \cCc(G,\mathcal L)$ with $a_n \to a$ and then choose an open subgroupoid $H \subseteq G$ covered by countably many bisections such that $a_n \in \cCc(H,\mathcal L|_H)$ for each $n$, so that $a \in \cs_\ess(H,\mathcal L|_H)$.
\end{remark}

\section{Essentially free measures}

Let $X$ be a locally compact Hausdorff space. A positive Borel measure on $X$ is said to be \emph{Radon} if it is finite on compact spaces, inner regular on open sets and outer regular on Borel sets. We shall identify, via the Riesz representation theorem \cite[Theorem 7.2]{FollandRealAnalysis}, the set of positive Radon measures on $X$ with the set of positive linear functionals on $C_c(X)$. Note that if $X$ is second countable, every positive Borel measure that is finite on compact sets is Radon \cite[Theorem 7.8]{FollandRealAnalysis}.

A finite complex Borel measure on $X$ is said to be Radon if its variation is Radon in the previous sense. We shall identify, via another version of the Riesz representation theorem (\cite[Theorem 7.3.6]{Coh13}), the set of finite complex Radon measures on $X$ with the dual of the space $C_c(X)$ endowed with the sup norm.

By default, we assume measures on locally compact Hausdorff spaces to be Borel and positive (but shall sometimes emphasise positivity).

The \emph{Pedersen ideal} of a $\cs$-algebra $A$, denoted by $\Ped(A)$, is the smallest dense ideal of $A$ (see \cite[Section 5.6]{Ped18}). Given a finite set $Y\subset \Ped(A)$, the hereditary $\cs$-subalgebra generated by $Y$ is contained in $\Ped(A)$. For $A = C_0(X)$, the Pedersen ideal is $C_c(X)$.

In this paper, by a \emph{trace} $\tau$ on $A$ we shall mean a positive tracial linear functional on $\Ped(A)$. When $A$ is unital, $\Ped(A) = A$ and all traces $\tau$ are bounded with norm $\tau(1)$. \begin{definition}
Let $G$ be an \'etale groupoid. A Radon measure $\mu$ on $G^{0}$ is \emph{invariant} if $\mu(\source(U))=\mu(\range(U))$ for every bisection $U \subset G$. 
\end{definition}

\begin{proposition}
Let $(S,X)$ be a pseudogroup. A Radon measure $\mu$ on $X$ is $G(S,X)$-invariant if and only if $\mu(V)=\mu(t(V))$ for each $t\in S$ and measurable subset $V\subset \dom t$.
\end{proposition}
\begin{proof}
The forward implication follows easily from outer regularity, so let us show the converse. 

Suppose that there exists a bisection $U\subset G(S,X)$ such that $\mu(\source(U))\neq \mu(\range(U))$. By inner regularity, there exists a compact set $K\subset U$ such that $\mu(\source(K))\neq \mu(\range(K))$. 
Take $t_1,\dots,t_n\in S$ such that $K\subset \bigcup_{i=1}^n\Theta_{t_i}$. 
Set $W_1:=K\cap\Theta_{t_1}$ and, for $1<i\leq n$, 
\[W_i:=K\cap\Theta_{t_i}\setminus \bigsqcup_{j=1}^{i-1}W_j.\]
Then $K=\bigsqcup_{i=1}^n W_i$ and $\mu(\source(W_i))=\mu(\range(W_i))$ for $1\leq i\leq n$, contradicting the fact that $\mu(\source(K))\neq \mu(\range(K))$.
\end{proof}

We record the following connection between invariant Radon measures and positive traces on groupoid $\cs$-algebras (\cite[Lemmas 7.1, 7.2 and 7.4]{GM26}).

\begin{proposition}\label{prop:inv}
	Let $(G,\mathcal L)$ be a twisted étale groupoid. Then
    \[\cCc(G,\mathcal L) \subset  \Ped(\cs(G,\mathcal L))\]    
and traces on $\cs(G,\mathcal L)$ are determined by their values on $\cCc(G,\mathcal L)$. Moreover, for each trace $\tau$ the restriction $\tau|_{C_c(G^0)} \colon C_c(G^0) \to \mathbb C$ is represented by an invariant Radon measure on $G^0$. 
\end{proposition}

Given a twisted étale groupoid $(G,\mathcal L)$ and an invariant Radon measure $\mu$ on $G^0$, we shall say that a trace $\sigma$ on $\cs(G,\cL)$ or on $\cs_\ess(G,\cL)$ is an \emph{extension} of $\mu$ if $\sigma|_{C_c(G^0)}$ is represented by $\mu$. For every invariant Radon measure $\mu$ on $G^0$, there is a (unique) trace $\tau_\mu$ on $\cs(G,\cL)$ such that $\tau_\mu(a)=\int a|_{G^0}\, \mathrm d\mu$ for each $a\in \cCc(G,\cL)$ \cite[Corollary 7.5]{GM26}. We shall call $\tau_\mu$ the \emph{canonical extension} of $\mu$.
\begin{remark}
    Given a $\cs$-algebra A and a trace $\tau$ on $A$ for which there exists an approximate unit $(e_i)\subset \Ped(A)$ such that $\sup\tau(e_i)<\infty$, we have that $\tau$ is bounded. Indeed, for $a \in \Ped(A)_+$, we have 
    \[\tau(a) = \lim_i \tau(a^{1/2} e_i a^{1/2}) = \lim_i \tau(e_i^{1/2} a e_i^{1/2}) \leq \sup_i \|a\| \tau(e_i),\] where the first equality follows from the fact that $a^{1/2}e_ia^{1/2}\in\overline{aAa}\subset\Ped(A)$. 
    
    Hence, given a twisted étale groupoid $(G,\cL)$ and $\sigma$ a trace on $\cs(G,\cL)$ or on $\cs_\ess(G,\cL)$ that extends an invariant Radon measure $\mu$, we have $\|\sigma\| = \mu(G^0) \in [0,\infty]$.
\end{remark}

Let $\mu$ be a measure on a locally compact Hausdorff set $X$. Recall that $\mu$ is said to be \emph{semifinite} if each Borel set with infinite measure contains a Borel set with finite nonzero measure; this is automatic if $\mu$ is $\sigma$-finite.
The \emph{semifinite part} $\mu_{\mathrm{sf}}$ of $\mu$ is a semifinite measure on $X$ given by 
\[\mu_{\mathrm{sf}}(A) = \sup \{ \mu(B) \mid B \subset A \text{ Borel, } \mu(B) < \infty \} \qquad \text{for $A \subset X$ Borel.} \]
It follows easily from \cite[Exercise 1.3.14]{FollandRealAnalysis} and \cite[Proposition 7.2.6]{Coh13} that the semifinite part $\mu_{\mathrm{sf}}$ of a Radon measure $\mu$ is inner regular on all Borel sets, but it may fail to be outer regular. 
Of course, $\mu_{\mathrm{sf}} = \mu$ if and only if $\mu$ is semifinite, thus the reader interested only in semifinite measures may ignore the subscripts. 
\begin{definition}\label{def:ess free}
Let $G$ be an étale groupoid and $\mu$ a Radon measure on $G^{0}$. We say that $G$ is \emph{essentially free with respect to $\mu$} if for each bisection $U \subset G$ we have 
\begin{equation*}
\mu_{\mathrm{sf}}(\source(U\cap\Iso\setminus G^{0}))=0.
\end{equation*}
When we want to emphasise the measure $\mu$, we say that $\mu$ is \emph{essentially free} (with respect to $G$).
\end{definition}

The idea of the notion of essential freeness is as follows: the bisection $U$ induces a homeomorphism $\theta_U=\range\circ (\source|_U)^{-1}\colon \source(U)\to \range(U)$. Note that $\Fix_{\theta_U}=\source(\Iso\cap U)=\range(\Iso\cap U)$. Obviously, each point in $U\cap G^0$ is a fixed point for $\theta_U$. Essential freeness requires that the set of points of $\source(U)$ which are not ``obviously'' fixed by $\theta_U$ has measure $0$. We work with the semifinite part of the measure in order to obtain the characterization of uniqueness of extensions that we are aiming for.

\begin{proposition}\label{prop:cover}
Let $G$ be an étale groupoid, $\mathcal{F}$ a cover of $G$ by bisections and $\mu$ a Radon measure on $G^0$. Then $\mu$ is essentially free if and only if $\mu_{\mathrm{sf}}(\source(U\cap\Iso\setminus G^{0}))=0$ for each $U\in\mathcal{F}$.
\end{proposition}
\begin{proof}
The forward implication is obvious, so let us show the converse. Suppose that there exists a bisection $V\subset G$ with $\mu_{\mathrm{sf}}(\source(V\cap\Iso\setminus G^{0}))>0$; we will show that there is $U\in\mathcal{F}$ such that $\mu_{\mathrm{sf}}(\source(U\cap\Iso\setminus G^{0}))>0$. 

Choose $K\subset \source(V\cap\Iso\setminus G^{0})$ compact such that $\mu(K)>0$. Take $U_1,\dots, U_n\in\mathcal{F}$ such that $(\source|_V)^{-1}(K)\subset \bigcup_{i=1}^nU_i$. Since $K\subset \bigcup_{i=1}^n \source(U_i\cap\Iso\setminus G^{0})$, the conclusion follows.
\end{proof}

The next result is an immediate consequence of Lemma \ref{eq:fixb} and Proposition \ref{prop:cover}.
\begin{proposition}\label{prop:pg} Let $(S,X)$ be a pseudogroup and $\mu$ a Radon measure on $X$. Then $G(S,X)$ is essentially free with respect to $\mu$ if and only if $\mu_{\mathrm{sf}}(\Fix_s\setminus\interior\Fix_s)=0$ for all $s\in S$.
\end{proposition}

\begin{remark}
	If $G$ is a Hausdorff étale groupoid, then a Radon measure $\mu$ on $G^0$ is essentially free if and only if $\mu_{\mathrm{sf}}(\source(U\cap \Iso))=0$ for any bisection $U\subset  G\setminus G^0$.
\end{remark}

The description $\source(\Iso\setminus G^0)=\{x\in G^0 \mid G_x^x\neq\{x\}\}$ yields the following immediate consequence.

\begin{proposition}\label{ess free measure condition}
Let $G$ be an étale groupoid which can be covered by countably many bisections. Then $G$ is essentially free with respect to a Radon measure $\mu$ on $G^0$ if and only if $\mu_{\mathrm{sf}}(\{x\in G^0 \mid G_x^x\neq\{x\}\})=0$. 
\end{proposition}

The following example illustrates the pathologies that can happen when the measure is not semifinite.
\begin{example}\label{non semifinite example}
Let $X = \mathbb R \times \mathbb R_{\mathrm d}$, where $\mathbb R_{\mathrm d}$ is the real numbers taken with the discrete topology, and let $\mu$ be a Haar measure on $X$. 

Consider the homeomorphism $s\colon (x,y)\mapsto (-x,y)$ on $X$. It is not difficult to see that $\mu$ is $s$-invariant. Furthermore, $\Fix_s = \{0\} \times \mathbb R_{\mathrm d}$ and, as observed in \cite[Section 2.3]{FollandAbstractHarmonicAnalysis}, outer regularity of $\mu$ implies that $\mu(\{0\} \times \mathbb R_{\mathrm d})=+\infty$. On the other hand, for any compact set $K\subset \{0\}\times \R_d$, we have $\mu(K)=0$. Thus, the transformation groupoid $\Z_2\ltimes X$ is essentially free with respect to $\mu$.
\end{example}

\begin{remark}\label{rmk:simple}
Suppose $G$ is essentially free with respect to a Radon measure $\mu$ on $G^0$ with full support. Then $G$ is topologically free in the sense that $\interior(\Iso\setminus G^0)=\emptyset$. Indeed, suppose that there exists a nonempty bisection $U \subset  \Iso\setminus G^0$. We may assume that $\source(U)$ has finite measure. Since $\mu$ has full support, we have $\mu(\source(U))>0$, which contradicts essential freeness of $\mu$. 
Under the additional assumption that $G$ is minimal (which implies that any invariant Radon measure has full support), the essential $\cs$-algebra $\cs_{\ess}(G,\mathcal L)$ is simple for any twist $\mathcal L$ over $G$ (see \cite[Theorem 7.26]{KM21}).
\end{remark}

\section{Traces on full groupoid $\cs$-algebras}\label{section:full}

Given a measurable space $X$, $\mu$ a positive measure on $X$ and $\nu$ a finite complex measure on $X$, recall that $\nu$ is said to be \emph{absolutely continuous} with respect to $\mu$, denoted by $\nu\ll\mu$, if each measurable set $A\subset  X$ that satisfies $\mu(A)=0$ also satisfies $\nu(A)=0$. 
A version of the next result was used in \cite{KTT} without proof. We provide a proof here for completeness. 
\begin{lemma}\label{lem:tec}
Let $X$ be a locally compact Hausdorff space, $\mu$ a positive Radon measure on $X$, and $\nu$ a finite complex Radon measure on $X$ such that all $f\in C_c(X)$ nonnegative satisfy $|\int f \, \mathrm d \nu|\leq (\int f^2 \, \mathrm d \mu)^{1/2}$. Then $\nu\ll\mu$.
\end{lemma}
\begin{proof}
By considering real and imaginary parts, it suffices to consider the case in which $\nu$ is signed. 

Suppose $A \subset  X$ is Borel with $\mu(A) = 0$ but $\nu(A) \ne 0$.
By \cite[Proposition 7.3.9]{Coh13}, there is $K\subset A$ compact such that $\mu(K)=0$ and $\nu(K)\neq 0$. 

Fix $\epsilon>0$. Choose $U\supset K$ open such that $\mu(U)<\epsilon$, $\nu^+(U)<\nu^+(K)+\epsilon$ and $\nu^-(U)<\nu^-(K)+\epsilon$.
Pick $f\in C_c(U)$ with $0 \leq f \leq 1$ and $f = 1$ on $K$. Then
\begin{align*}
|\nu(K)|&=|\nu^+(K)-\nu^-(K)|\\
&\leq\left|\nu^+(K)-\int f \, \mathrm d \nu^+ \right|+\left|\int f \, \mathrm d \nu^+-\int f \, \mathrm d \nu^-\right|+\left|\int f \, \mathrm d \nu^- -\nu^-(K)\right|\\
&< \epsilon+\left(\int f^2 \, \mathrm d \mu\right)^{1/2} +\epsilon <2\epsilon+\epsilon^{1/2}.
\end{align*}
Thus, $\nu\ll\mu$.
\end{proof}

\begin{lemma}\label{lem:fiso}
Let $(G,\mathcal L)$ be a twisted étale groupoid, $\tau$ a trace on $\cs(G,\mathcal L)$ and $U\subset G$ a bisection with $U\cap \Iso=\emptyset$. Then for any $f\in C_c(U,\mathcal L)$, we have $\tau(f)=0$.
\end{lemma}
\begin{proof}
Since we can cover $\supp f$ with bisections $V_1,\dots, V_n$ such that, for all $i$, $\source(V_i)\cap \range(V_i)=\emptyset$, by a standard partition of unity argument we can assume that $\source(U)\cap \range(U)=\emptyset$.

Take $h\in C_c(G^{0})$ which is $1$ on $\source(\supp f)$ and $0$ on $\range(\supp f)$. Then $fh=f$ and $hf=0$, hence $\tau(f)=\tau(fh)=\tau(hf)=0$.
\end{proof}

In the proof of the next result, we follow an idea from \cite{KTT}.
\begin{theorem}\label{thm:u} 
Let $(G,\mathcal L)$ be a twisted étale groupoid and let $\mu$ be an essentially free invariant Radon measure on $G^0$. Then the only extension of $\mu$ to a trace on $\cs(G,\mathcal L)$ is the canonical one.
\end{theorem}

\begin{proof}
Let $\tau$ be a trace on $\cs(G,\mathcal L)$ extending $\mu$.  
By Proposition \ref{prop:trivial} we may pick an open cover $\mathcal U$ of $G$ by trivialisable bisections $U$ such that $\source(U) \cup \range(U)\subset G^0$ is relatively compact. By Proposition \ref{prop:covera}, the groupoid algebra $\cCc(G,\mathcal L)$ is spanned by the subspaces $C_c(U,\mathcal L)$ for $U \in \mathcal U$. By Proposition \ref{prop:inv}, the trace $\tau$ is determined by its values on the groupoid algebra $\cCc(G,\mathcal L)$, so it suffices to show for each $U \in \mathcal U$ and $a \in C_c(U,\mathcal L)$ that $\tau(a) = \int_{G^0} a|_{G^0} \, \mathrm d \mu$. 

The trace $\tau$ is bounded on $C_c(U,\mathcal L)\subset\cs(G,\cL)$ because, picking $h \in C_c(G^0)$ nonnegative with $h = 1$ on $\source(U)\cup\range(U)$, $C_c(U,\mathcal L)$ is contained in the $\cs$-subalgebra $\overline{h \cs(G,\mathcal L) h} \subset \Ped(\cs(G,\mathcal L))$. 

Pick a trivialisation $T \colon U \times \mathbb C \to \mathcal L|_U$. Let $\Phi_U \colon C_c(U) \to C_c(U,\mathcal L)$ be the isometric isomorphism given by $\Phi_U(f) \colon u \mapsto T(u,f(u))$.
Thus by the Riesz representation theorem the functional $\tau \circ \Phi_U \colon C_c(U) \to \mathbb C$ defines a finite complex Radon measure $\nu$ on $U$. 

Under the identification of a fibre $\cL_u$ with $\C$, $u\in G^0$, the trivialisation $T$ induces a continuous function $t \colon U \cap G^0 \to \mathbb T$ such that $t(u) = T(u,1)$. Hence, given $f\in C_c(U)$ and $u\in U \cap G^0$, we have that $\Phi_U(f)(u)=T(u,f(u))=t(u)f(u)$. 

Given $g\in C_c(U \cap G^0)$, we have that $\int_{U \cap G^0}g\,\mathrm d\nu=\tau(\Phi_U(g))=\int_{U \cap G^0}tg\, \mathrm d\mu$. An application of the Riesz representation theorem and \cite[Proposition 7.3.7]{Coh13} shows then that $\int_{U \cap G^0}g\,\mathrm d\nu=\int_{U \cap G^0}tg\,\mathrm d\mu$ for any $g\in B_\infty(U \cap G^0)$.

For any $f \in C_c(U \setminus \Iso)$, Lemma \ref{lem:fiso} implies that $\tau(\Phi_U(f)) = 0$. It follows that the restriction Radon measure $\nu|_{U \setminus \Iso}$ on $U \setminus \Iso$ vanishes.

We now claim that $\nu \ll (\source|_U^{-1})_* \mu|_{\source(U)}$. Let $R$ be the norm of the restriction of $\tau$ to $\cs(C_c(U,\mathcal L))$. By Lemma \ref{lem:tec}, it is enough to show that $|\int_U f \, \mathrm d \nu|^2 \leq R \int_{\source(U)} (f \circ \source|_U^{-1})^2 \, \mathrm d \mu$ for each $f \in C_c(U)$ nonnegative. Indeed, setting $a = \Phi_U(f)$ and noting that $a^* a = (f \circ \source|_U^{-1})^2 \in C_c(\source(U))$, we see that 
\[ \left\lvert \int_U f \, \mathrm d \nu \right\rvert^2 = |\tau(a)|^2 \leq R \tau(a^*a) =  R \int_{\source(U)}  (f \circ \source|_U^{-1})^2 \, \mathrm d \mu. \]
Essential freeness implies $\mu(\source(U \cap \Iso \setminus G^0)) = 0$, so we have $(\source|_U^{-1})_* \mu|_{\source(U)}(U \cap \Iso \setminus G^0) = 0$. By absolute continuity, any integral over $U \cap \Iso \setminus G^0$ with respect to $\nu$ must vanish.

Finally, let $a \in \cCc(U,\mathcal L)$ be arbitrary and let $f = \Phi_U^{-1}(a) \in C_c(U)$. Combining the above observations, we compute
\begin{align*}
\tau(a) & = \int_U f \, \mathrm d\nu \\
& = \int_{U \cap G^0} f \,\mathrm d\nu + \int_{U\setminus\Iso} f|_{U \setminus \Iso} \,\mathrm d\nu + \int_{U\cap\Iso\setminus  G^0} f|_{U \cap \Iso \setminus G^0} \,\mathrm d\nu \\
&= \int_{U \cap G^0} tf \,\mathrm d\mu \\
&=\int a|_{G^0}\,\mathrm d\mu. \qedhere
\end{align*}
\end{proof}

We will now show the converse of Theorem \ref{thm:u} in the untwisted setting.

 \begin{lemma}\label{prop:ev}
Given an étale groupoid $G$, $x\in G^0$ and $H\leq G_x^x$, there is a $*$-representation $L_{G_x/H}\colon \cs(G)\to \mathcal B(\ell^2(G_x/H))$ such that 
\begin{equation}\label{eq:lambda}
L_{G_x/H}(a)(\delta_{gH})=\sum_{k\in G_{\range(g)}}a(k)\delta_{kgH},
\end{equation}
for $a\in \cCc(G)$ and $g\in G_x$. Moreover, there is a state $\varphi$ on $\cs(G)$ such that $\varphi(a)=\sum_{h\in H}a(h)$ for any $a\in \cCc(G)$, and $L_{G_x/H}$ is unitarily equivalent to the GNS representation of $\varphi$. 
\end{lemma}
\begin{proof}
Given $U\subset G$ a bisection and $a\in C_c(U)$, it is easy to see that $\eqref{eq:lambda}$ defines a bounded operator on $\ell^2(G_x/H)$, hence the same holds for any $a\in \cCc(G)$.

Let us check that $L_{G_x/H}$ is a $*$-homomorphism on $\cCc(G)$. Given $a\in \cCc(G)$ and $g_1,g_2\in G_x$, we have that
\[\langle L_{G_x/H}(a)\delta_{g_1H},\delta_{g_2H}\rangle=\sum_{k\in G_{\range(g_1)}}a(k)\langle \delta_{kg_1H},\delta_{g_2H}\rangle=\sum_{h\in H}a(g_2hg_1^{-1})\]
and
\[\langle\delta_{g_1H},L_{G_x/H}(a^*)\delta_{g_2H}\rangle=\sum_{h\in H}\overline{a^*(g_1hg_2^{-1})}=\sum_{h\in H}a(g_2hg_1^{-1}),\]
thus showing that $L_{G_x/H}(a^*)=L_{G_x/H}(a)^*$.

Given $a,b\in \cCc(G)$, and $g\in G_x$, we have that
\begin{align*}
L_{G_x/H}(ab)\delta_{gH}=\sum_{k\in G_{\range(g)}}(ab)(k)\delta_{kgH}&=\sum_{k\in G_{\range(g)}}\sum_{l\in G_{\range(g)}}a(kl^{-1})b(l)\delta_{kgH}\\
&=\sum_{l\in G_{\range(g)}}b(l)\sum_{k\in G_{\range(g)}}a(kl^{-1})\delta_{kgH}\\
&=\sum_{l\in G_{\range(g)}}b(l)\sum_{t\in G_{\range(l)}}a(t)\delta_{tlgH}.
\end{align*}
Moreover,
\[L_{G_x/H}(a)L_{G_x/H}(b)\delta_{gH}=\sum_{l\in G_{\range(g)}}b(l)L_{G_x/H}(a)\delta_{lgH}=\sum_{l\in G_{\range(g)}}b(l)\sum_{t\in G_{\range(l)}}a(t)\delta_{tlgH},\]
thus showing that $L_{G_x/H}(ab)=L_{G_x/H}(a)L_{G_x/H}(b)$.

The last claims of the proposition follow from \cite[Theorem 5.1.7]{M90} and the facts that $\delta_H$ is cyclic for $L_{G_x/H}$ and $\langle L_{G_x/H}(a)\delta_H,\delta_H\rangle = \sum_{h\in H}a(h)$ for any $a\in \cCc(G)$.
\end{proof}

\begin{lemma}\label{lem:tecb}
Let $G$ be an étale groupoid. Given $a\in \cCc(G)$, the map $F(a)\colon G^0\to\C$ given by $F(a)(x):=\sum_{g\in G_x^x}a(g)$ is Borel measurable.
\end{lemma}
\begin{proof}
It suffices by linearity to verify the claim for $a\in C_c(U)$, where $U\subset G$ is a bisection. In this case, the function $F(a)=(a\circ (\source|_U)^{-1})1_{\source(U\cap \Iso)}$ is Borel.
\end{proof}

\begin{theorem}\label{thm:nu}
	Let $G$ be an étale groupoid and $\mu$ an invariant Radon measure on $G^0$. Then there is a trace $\varphi$ on $\cs(G)$ which satisfies 
	\[\varphi(f)=\int\sum_{g\in G_x^x}f(g)\,\mathrm d\mu(x)\]
	for each $f\in \cCc(G)$. Furthermore, if $\mu$ is not essentially free, then $\varphi$ is an extension of $\mu$ which is different from the canonical one.
\end{theorem}
\begin{proof}
By Lemmas \ref{prop:ev} and \ref{lem:tecb}, there exists a contractive positive linear map $F\colon \cs(G)\to B_{\infty}(G^{0})$ that satisfies $F(a)(x)=\sum_{g \in G_x^x}a(g)$ for each $a\in \cCc(G)$ and $x\in G^{0}$. 

Let $\varphi$ be the positive linear functional on $\cCc(G)$ given by $\varphi(a) = \int F(a)\, \mathrm d\mu$ for $a \in \cCc(G)$. By \cite[Proposition 7.2]{GM26}, this extends uniquely to a positive trace on $\cs(G)$ if and only if it is tracial and it is bounded on $\cCc(G)|_Y = \{f \in \cCc(G) \mid f = 0 \text{ outside } G^Y_Y\}$ for each relatively compact open $Y \subset  G^0$. The boundedness condition follows from contractivity of $F$, that $F(\cCc(G)|_Y) \subset  B_{\infty}(Y)$, and that $\tau_\mu$ is bounded on $B_{\infty}(Y)$.

Let us show that $\varphi$ is tracial. 
Take bisections $U,V\subset  G$ and functions $f\in C_c(U)$ and $g\in C_c(V)$. Then
	\begin{align*}F(fg)(x)=
		\begin{cases}
			f((\range|_U)^{-1}(x))g((\source|_V)^{-1}(x)),&\text{ if $x\in \source(UV)$, $(\source|_{UV})^{-1}(x)\in\Iso$}\\
			0,&\text{ otherwise,} 
	\end{cases}\end{align*}
	and
	\begin{align*}F(gf)(x)=
		\begin{cases}
			g((\range|_V)^{-1}(x))f((\source|_U)^{-1}(x)),&\text{ if $x\in \source(VU)$, $(\source|_{VU})^{-1}(x)\in\Iso$}\\
			0,&\text{ otherwise.} 
	\end{cases}\end{align*}
	Given $x\in \source(UV)$ such that $(\source|_{UV})^{-1}(x)\in\Iso$, there exists $h\in U$ and $k\in V$ such that $\source(h)=\range(k)$ and $\range(h)=\source(k)=x$. Then 
	\[F(fg)(x)=f(h)g(k)=g(k)f(h)=F(gf)(\range(k))=F(gf)((\range\circ (\source|_V)^{-1})(x)).\]
	By invariance of $\mu$, we conclude that $\varphi(fg)=\varphi(gf)$ and hence $\varphi$ is tracial.

Suppose that $\mu$ is not essentially free, and take a bisection $U\subset  G$ such that $\mu_{\mathrm{sf}}(\source(U\cap\Iso\setminus G^{0}))>0$. There exists a compact set $K\subset \source(U\cap\Iso\setminus G^{0})$ such that $\mu(K)>0$. Let $f\in C_c(U)$ such that $0\leq f\leq 1$ and $f = 1$ on $(\source|_U)^{-1}(K)$.
Then $\varphi(f)=\int F(f)\,\mathrm d\mu>\int E(f)\,\mathrm d \mu$, hence $\varphi$ is different from the canonical extension of $\mu$.
\end{proof}

\begin{remark}
The converse to Theorem \ref{thm:u} does not hold in general for twisted \'etale groupoids. For example, the irrational rotation algebra $A_\theta$ has a well-known realisation as the twisted group $\cs$-algebra of a twist over $\mathbb Z^2$, and is known to have a unique tracial state. The point-evaluation on the unit $e \in \mathbb Z^2$ can therefore only have one extension to a trace on $A_\theta$, despite not being essentially free.
\end{remark}

\begin{remark}
We do not discuss here the question of uniqueness of extension of an invariant Radon measure to a trace on the reduced (or the essential) $\cs$-algebra of an étale groupoid. In this direction, the sharpest results are \cite[Corollary 4.3]{BKKO} for groups and, more generally, \cite[Corollary 1.12]{Urs21} for transformation groupoids with compact unit spaces. 
\end{remark}

\section{Existence of traces on the essential C*-algebra}\label{section:ess}

In this section, we provide criteria ensuring that an invariant Radon measure on the unit space of an étale groupoid $G$ can be extended to a trace on $\cs_\ess(G)$. The canonical trace $\tau_\mu$ on $\cs(G)$ associated to an invariant Radon measure $\mu$ does not always descend to the quotient $\cs_\ess(G)$ (see \cite[Remark 7.7]{GM26}).

The following two lemmas are useful for determining when traces descend to the essential $\cs$-algebra.

\begin{lemma}\label{lem:descend}
Let $\pi \colon A \to B$ be a surjective $*$-homomorphism and let $\tau$ be a trace on $A$. If $\tau(a) = 0$ for each $a \in \Ped(A)_+ \cap \ker \pi$, then there is a unique trace $\sigma$ on $B$ such that $\sigma \circ \pi = \tau$.
\end{lemma}
\begin{proof}
The map $\pi \colon \Ped(A) \to \Ped(B)$ is surjective because $\pi(\Ped(A)) \subset B$ is a dense ideal, hence contains $\Ped(B)$. Let $a \in \Ped(A) \cap \ker \pi$ be arbitrary. As $\tau$ is bounded on $\cs(a) \subset \Ped(A)$ and $\tau(a^* a) = 0$ we conclude $\tau(a) = 0$. It follows that the positive tracial linear functional $\sigma \colon \Ped(B) \to \mathbb C$ given by $\sigma(\pi(a)) = \tau(a)$ for $a \in \Ped(A)$ is well-defined. 
\end{proof}

The following idea is also used in \cite[Corollary 7.8]{GM26}.
\begin{lemma}\label{lem:groupoid ped}
Let $(G,\mathcal L)$ be a twisted \'etale groupoid and let $a \in \Ped(\cs(G,\mathcal L))_+$. Then there are $f \in C_c(G^0)_+$ and $c \in \cs(G,\mathcal L)_+$ such that $\tau(fcf) = \tau(a)$ for any trace $\tau$ on $\cs(G,\mathcal L)$ and, for any closed ideal $I \vartriangleleft \cs(G,\mathcal L)$, $fcf \in I$ if and only if $a \in I$.
\end{lemma}
\begin{proof}
By \cite[Lemma 5.2]{GM26}, there are $f_1,\dots,f_k \in C_c(G^0)$ and $c_1,\dots,c_k \in \cs(G,\mathcal L)$ such that $a = \sum_{i=1}^k c_i f_i f_i^* c_i^*$. Set $c = \sum_{i=1}^k f_i^* c_i^* c_i f_i$ and pick $f \in C_c(G^0)$ with $0 \leq f \leq 1$ such that $f = 1$ on $\bigcup_{i=1}^k \supp f_i$, so that $c = fcf$. Clearly $\tau(a) = \tau(fcf)$ for any trace $\tau$. Membership of an ideal $I$ passes between $a$ and $c$ by hereditariness of $I$ and the $\cs$-identity within $\cs(G,\mathcal L)/I$.
\end{proof}

A unit $y\in G^0$ is said to be \emph{dangerous} if there exists a net $(g_i)\subset G$ such that $g_i\to y$ and $g_i\to h\neq y$. We denote by $D^0$ the set of dangerous units. Given any twist $\mathcal L$ over $G$, a section $\xi \in \cCc(G,\mathcal L)$ is continuous at every $y \in G^0 \setminus D^0$.

\begin{lemma}\label{point evaluation state}
Let $(G,\mathcal L)$ be a twisted \'etale groupoid and let $x \in G^0 \setminus D^0$ be a non-dangerous unit. Then there is a state $\varphi$ on $\cs_\ess(G,\mathcal L)$ such that $\varphi([a])=a(x)$ for any $a\in \cCc(G,\mathcal L)$. 
\end{lemma}
\begin{proof}
Fix $a\in \cs(G,\mathcal L)$. The function $E(a) \in B_\infty(G^0)$ is continuous at $x$. Suppose that $E(a)(x)\neq 0$. Then there is a neighbourhood $U$ of $x$ such that, for each $y\in U$, $E(a)(y)\neq 0$. By \cite[Theorem 3.3.2]{M90}, we conclude that $E(a^*a)(y)>0$ for each $y\in U$, hence $E(a^*a)\notin M(G^0)$ and $a\notin J$. The state $a \mapsto E(a)(x)$ on $\cs(G,\mathcal L)$ therefore vanishes on $J$ and the state $\varphi \colon [a] \mapsto E(a)(x)$ is well-defined on $\cs_\ess(G,\mathcal L)$.
\end{proof}

If $G$ can be covered by countably many bisections, then $D^0$ is meagre in $G^0$ (see, for example, \cite[Lemma 3.14]{BM26}) and therefore has dense complement.

\begin{proposition}\label{prop:lit}
Let $G$ be an étale groupoid. Then $D^0=\source(\overline{G^0}\setminus G^0)$. Moreover, if $G$ can be covered by countably many bisections, then $D^0$ is a Borel subset of $G^0$.
\end{proposition}
\begin{proof}
Let us begin by showing that $D^0\subset  \source(\overline{G^0}\setminus G^0)$. Given $y\in D^0$, take a net $(g_i)\subset G$ such that $g_i\to y$ and $g_i\to g\neq y$. Then $g_ig_i^{-1}\to gy=g$. Hence $y=\source(g)\in \source(\overline{G^0}\setminus G^0)$. The proof of the reverse inclusion is straightforward.

The last claim follows from the fact that the restriction of $\source$ to a bisection is a homeomorphism onto its image, and $\overline{G^0}\setminus G^0$ is a closed subset of $G$.
\end{proof}

\begin{theorem}\label{thm:sum}
Let $G$ be an étale groupoid that can be covered by countably many bisections. Given $x\in G^0$, there exists $H\leq G_x^x$ and a state $\varphi$ on $\cs_{\ess}(G)$ such that $\varphi([a])=\sum_{h\in H}a(h)$ for any $a\in \cCc(G)$. In particular, $\|L_{G_x/H}(a)\| \leq \|[a]\|$.
\end{theorem}
\begin{proof}
Take an ultranet $(x_\lambda)\subset G^0\setminus D^0$ such that $x_\lambda \to x$. Set $H:=\{h\in G \mid x_\lambda\to h\}$. By continuity of the multiplication and inversion operations, we have that $H\leq G_x^x$ (this was also observed in \cite{Hum25}).

For each $\lambda$, let $\varphi_{\lambda}$ be the state on $\cs_{\ess}(G)$ guaranteed by Lemma \ref{point evaluation state} with $\varphi_{\lambda}([a])=a(x_\lambda)$ for any $a\in\cCc(G)$.
We claim that, given $a\in \cCc(G)$, 
\begin{equation}\label{eq:func}
a(x_\lambda)\to \sum_{h\in H}a(h).
\end{equation}
By linearity, we can assume that $a\in C_c(U)$, where $U\subset G$ is a bisection. If there exists $g\in U\cap H$, then continuity of $a$ implies that $a(x_\lambda)\to a(g)=\sum_{h\in H} a(h)$.

Suppose that the right-hand side of \eqref{eq:func} is nonzero. Then there exists $g\in H$ such that $a(g)\neq 0$, which implies that $H\cap U\neq \emptyset$, thus showing \eqref{eq:func} in this case. 

Suppose instead that the right-hand side of \eqref{eq:func} is zero. Let $K \subset U$ be the compact support of $a$. Since $(x_\lambda)$ is an ultranet, it is eventually contained in either $K$ or the complement of $K$. In the first case it converges to some $g \in K$, which implies that $g \in H \cap U$. In the second case $a(x_\lambda) = 0$ for sufficiently large $\lambda$. Both cases yield \eqref{eq:func}.

Since $\cCc(G)$ is dense in $\cs_\ess(G)$, it follows that, for every $a\in \cs_\ess(G)$, the net $\varphi_\lambda(a)$ is convergent. 

Let $\varphi=\lim_\lambda \varphi_\lambda$. Clearly, $\varphi$ is a positive linear functional on $\cs(G)$. Furthermore, given $a\in C_c(G^0)$, we have that $\varphi(a)=a(x)$. Therefore, $\varphi$ is a state.

The final inequality follows from Lemma \ref{prop:ev}.
\end{proof}

Recall that a subgroup $H$ of a group $K$ is said to be co-amenable in K if $K\act K/H$ is an amenable action. Equivalently, the quasi-regular representation of $K$ on $\ell^2(K/H)$ weakly contains the trivial representation. The proof of the following result is inspired by the argument of \cite[Proposition 2.2]{KS22}. 

\begin{lemma}\label{lem:inc}
Let $G$ be an étale groupoid, $x\in G^0$ and $H\leq G_x^x$ such that $H$ is co-amenable in $G_x^x$. Then $L_{G_x/G_x^x}\prec L_{G_x/H}$.
\end{lemma}
\begin{proof}
Let $\lambda_{G_x^x/H}$ be the quasi-regular representation of $G_x^x$ on $\ell^2(G_x^x/H)$ and $P\in B(\ell^2(G_x/H))$ be the orthogonal projection on $\ell^2(G_x^x/H)$. 

Given $U\subset G$ a bisection and $a\in C_c(U)$, a straightforward computation shows that
\[PL_{G_x/H}(a)P=\begin{cases}
a(k)\lambda_{G_x^x/H}(k),&\text{if $U\cap G_x^x=\{k\}$}\\
0,&\text{if $U\cap G_x^x=\emptyset$}
\end{cases}\]
Hence, given $a\in \cs_{\la_{G_x/H}}(G)$, we have $PaP\in \cs_{\lambda_{G_x^x/H}}(G_x^x)$.

By co-amenability, there is a state $\theta$ on $\cs_{\lambda_{G_x^x/H}}(G_x^x)$ such that $\theta(\lambda_{G_x^x/H}(g))=1$ for every $g\in G_x^x$. Let $\psi$ be the state on $\cs_{\la_{G_x/H}}(G)$ given by $\psi(a)=\theta(PaP)$, for $a\in \cs_{\la_{G_x/H}}(G)$. 

Then, given $a\in \cCc(G)$, we have that $\psi(\la_{G_x/H}(a))=\sum_{g \in G_x^x}a(g)$. By applying Lemma \ref{prop:ev}, we conclude that $L_{G_x/G_x^x}\prec L_{G_x/H}$.
\end{proof}

\begin{theorem}\label{thm:new}
Let $G$ be an étale groupoid and $X\subset G^0$ a set such that $G_x^x$ is amenable for each $x\in X$. There is a contractive positive linear map $F\colon \cs_\ess(G)\to B_\infty(X)$ such that, for every $a\in  \cCc(G)$ and $x\in X$,
\begin{equation}\label{eq:amen}
F([a])(x)=\sum_{h\in G_x^x}a(h).
\end{equation}
\end{theorem}
\begin{proof}
We may assume that $G$ can be covered by countably many bisections using Remark \ref{countable remark}, because the formula \eqref{eq:amen} is consistent with open subgroupoids.

It follows from Theorem \ref{thm:sum} and Lemma \ref{lem:inc} that there is a contractive positive linear map $F\colon \cs_\ess(G)\to \ell^{\infty}(X)$ such that \eqref{eq:amen} holds for every $a\in \cCc(G)$ and $x \in X$. Furthermore, it is a consequence of Lemma \ref{lem:tecb} that $F(\cs_\ess(G))\subset B_\infty(X)$.
\end{proof}

\begin{corollary}\label{cor:amenable isotropy}
Let $G$ be an étale groupoid and let $\mu$ be an invariant Radon measure on $G^0$ such that $G_x^x$ is amenable almost everywhere. Then there exists a trace $\varphi$ on $\cs_\ess(G)$ which extends $\mu$ and satisfies \[\varphi([f])=\int\sum_{g\in G_x^x}f(g)\, \mathrm d\mu(x)\]
	for each $f\in \cCc(G)$.
\end{corollary}
\begin{proof}
Take $X\subset G^0$ a Borel set such that $\mu(X^c)=0$ and $G_x^x$ is amenable for every $x\in X$, and let $F\colon \cs_{\ess}(G)\to B_{\infty}(X)$ be as in Theorem \ref{thm:new}. 

By Lemma \ref{lem:descend}, we need only check that the trace $\widetilde \varphi$ from Theorem \ref{thm:nu} on $\cs(G)$ satisfies $\widetilde \varphi(a) = 0$ for each $a \in \Ped(\cs(G))_+ \cap J$. By Lemma \ref{lem:groupoid ped}, we may assume that $a = fcf$ for elements $f \in C_c(G^0)_+$ and $c \in \cs(G)_+$.
Pick $c_n \in \cCc(G)$ with $c_n \to c$, so that $f c_n f \to a$ in $\overline{f \cs(G) f}$ and $[f c_n f] \to 0$. Then $\widetilde \varphi(f c_n f) \to \widetilde \varphi(a)$. But, setting $K = \supp f$, for each $n$ we have $\widetilde \varphi(f c_n f) = \int_{X \cap K} F([f c_n f]) \, \mathrm d \mu$, which converges to $0$ by Theorem \ref{thm:new}.
\end{proof}

\begin{remark}
If $G$ is an amenable étale groupoid, then every isotropy group of $G$ is amenable (\cite[Proposition 5.6]{BM26}).
\end{remark}

If $\mu$ is essentially free, the canonical trace $\tau_\mu$ descends to the essential $\cs$-algebra:

\begin{proposition}\label{prop:efe}
Let $(G,\mathcal L)$ be a twisted étale groupoid and $\mu$ an essentially free invariant Radon measure on $G^{0}$. Then $\mu$ admits a unique extension to a trace $\tau$ on $\cs_\ess(G,\mathcal L)$. Furthermore, for every $\xi\in \cCc(G,\mathcal L)$, $\tau([\xi])=\int_{G^{0}}\xi|_{G^0}\,\mathrm d\mu$.
\end{proposition}
\begin{proof}
To build $\tau$ we argue via Lemma \ref{lem:descend} that the canonical extension $\tau_\mu$ on $\cs(G,\mathcal L)$ descends to $\cs_\ess(G,\mathcal L)$. Uniqueness then follows from Theorem \ref{thm:u}.

Let $a \in \Ped(\cs(G,\mathcal L))_+ \cap J$; we must show that $\tau_\mu(a) = 0$. By Lemma \ref{lem:groupoid ped}, it suffices to assume that $a = fcf$, where $f \in C_c(G^0)_+$ and $c \in \cs(G,\mathcal L)_+$. Pick $c_n \in \cCc(G, \mathcal L)$ with $c_n \to c$, and let $H \subset G$ be an open subgroupoid containing $G^0$ which is covered by countably many bisections such that $c_n \in \cCc(H,\mathcal L)$ for each $n$. 

Let $Y := \source(H \cap \Iso \setminus G^0)$, which is Borel, contains the set $D^0_H$ of units which are dangerous with respect to $H$ and satisfies $\mu_{\mathrm{sf}}(Y) = 0$ by essential freeness. Set $K := \supp f$. By Lemma \ref{point evaluation state} we obtain a contractive positive linear map $P \colon \cs_\ess(H,\mathcal L) \to B_\infty(K \setminus Y)$ such that $P([b]) = b|_{K \setminus Y}$ for each $b \in \cCc(H,\mathcal L)$. Since $f c_n f \to a$ in $\overline{f \cs(G,\mathcal L)f}$ we have $\tau_\mu(f c_n f) \to \tau_\mu(a)$. However, since $\mu(K \cap Y) = 0$ and $[f c_n f] \to 0$, then we also have \[\tau_\mu(f c_n f) = \int_{G^0} (f c_n f)|_{G^0} \, \mathrm d \mu =\int_K(fc_nf)|_K\,\mathrm d\mu= \int_{K \setminus Y} P([f c_n f]) \, \mathrm d \mu \to 0. \qedhere\] 
\end{proof}

\begin{remark}
For existence in Proposition \ref{prop:efe}, instead of assuming that $\mu$ is essentially free, we could have assumed the weaker condition that $\mu_{\mathrm{sf}}(D^0_H)=0$ for any open subgroupoid $H \subset G$ which contains $G^0$ and is covered by countably many bisections (by Proposition \ref{prop:lit}, $D^0_H$ is a Borel set). 
\end{remark}

\begin{remark}[See Proposition 3.14 in \cite{KMP}] 
Let $(G,\mathcal L)$ be a twisted \'etale groupoid and let $\mu$ be an invariant Radon measure on $G^0$ with full support. If the canonical trace $\tau_\mu$ descends to $\cs_\ess(G,\mathcal L)$, then (the descent of) $\tau_\mu$ is faithful on $\cs_\ess(G,\mathcal L)$. Indeed, let $a \in \Ped(\cs(G,\mathcal L))_+$ satisfy $\tau_\mu(a) = 0$; we will show that $a \in J$. By Lemma \ref{lem:groupoid ped} it suffices to show that $fcf \in J$ for $f \in C_c(G^0)_+$ and $c \in \cs(G,\mathcal L)_+$ with $\tau_\mu(fcf) = 0$. Since $\tau_\mu(fcf) = \int E(fcf) \, \mathrm d \mu$, the set $\{x \in G^0 \mid E(fcf) > 0 \}$ must have measure zero and hence empty interior. Thus by \cite[Proposition 4.6]{BKM} we conclude $fcf \in J$.

Alternatively, if $G$ is topologically free then any trace $\tau$ on $\cs_\ess(G,\mathcal L)$ extending $\mu$ is faithful. This follows from \cite[Theorem 7.25]{KM21}.
\end{remark}

For a $\cs$-algebra $A$ we topologise the space $T_+(A)$ of (positive) traces with the weak topology with respect to $\Ped(A)$. Similarly, for an \'etale groupoid $G$ the space $\Rad(G)$ of invariant Radon measures on $G^0$ carries the weak topology with respect to $C_c(G^0)$. Both spaces moreover have the structure of positive cones.

\begin{corollary}\label{cor:trace space}
Let $(G,\mathcal L)$ be a twisted \'etale groupoid such that $G$ is essentially free with respect to every invariant Radon measure on $G^0$. 
Then the restriction map \[\Phi \colon T_{+}(\cs_\ess(G,\mathcal L)) \to \Rad(G)\] induced by the inclusion $C_c(G^0) \subset \Ped(\cs_\ess(G,\mathcal L))$ is an affine homeomorphism.
\end{corollary}
\begin{proof}
The continuous affine function $\Phi$ is bijective by Theorem \ref{thm:u} and Proposition \ref{prop:efe}. 
Let us argue that $\Phi^{-1}$ is continuous. 

Set $A = \cs_\ess(G,\mathcal L)$ and let $\mu_i \to \mu$ be a convergent ultranet of invariant Radon measures on $G^0$, and consider the traces $\tau_i := \Phi^{-1}(\mu_i)$ and $\tau := \Phi^{-1}(\mu)$ on $A$. Our aim is to show that $\tau_i \to \tau$ in $T_+(A)$. 

Let $a \in \Ped(A)$ and consider the ultranet $(\tau_i(a))$; we claim it is bounded for sufficiently large $i$. To this end it suffices to assume $a$ is positive, and by Lemma \ref{lem:groupoid ped} we may moreover assume that $a = fcf$ for elements $f \in C_c(G^0)_+$ and $c \in A_+$.
The net $(\tau_i(fcf))$ is eventually bounded because $\tau_i(fcf) \leq \|c\| \tau_i(f^2) = \|c\| \int f^2 \, \mathrm d \mu_i \to \|c\| \int f^2 \, \mathrm d \mu$.

As the ultranet $(\tau_i(a))$ is eventually bounded for each $a \in \Ped(A)$, it must converge to some number $\sigma(a)$. The resulting functional $\sigma \colon \Ped(A) \to \mathbb C$ is a positive linear tracial functional. Moreover, $\tau_i \to \sigma$ by construction. The trace $\sigma$ extends $\mu$, thus by uniqueness of extensions (Theorem \ref{thm:u}) we deduce $\tau_i \to \tau$. 
\end{proof}

\begin{remark}
In the above corollary we could instead have considered any intermediate quotient $\cs$-algebra $\cs(G,\mathcal L) \twoheadrightarrow A \twoheadrightarrow \cs_\ess(G,\mathcal L)$.
\end{remark}

\begin{remark}
A \emph{self-adjoint trace} on a $\cs$-algebra $A$ is a functional on $\Ped(A)$ that can be expressed as the difference between two (positive) traces. The space $T_{\mathbb R}(A)$ of self-adjoint traces is studied in \cite{GM26}, where it is considered as a real topological vector space with an order. It is given the weak topology with respect to $\Ped(A)$. 
The space $T_{\mathbb R}(A)$ contains a priori more information than $T_+(A)$, but can be recovered from $T_+(A)$ for a large class of $\cs$-algebras (see \cite[Question A, Corollary 6.5]{GM26}).
For Hausdorff $G$, the conclusion in Corollary \ref{cor:trace space} can be upgraded to an identification $T_{\mathbb R}(\cs_r(G,\mathcal L)) \cong \Rad_{\mathbb R}(G)$ by \cite[Corollary 7.9]{GM26}, where the space $\Rad_{\mathbb R}(G)$ of invariant real Radon measures is taken with the weak topology with respect to $C_c(G^0)$.
\end{remark}

Recall that a \emph{group bundle} is a groupoid $G$ such that $G=\Iso$. This class of groupoids has been used many times in the literature as a source of counterexamples (see \cite[Section 4]{MS26} for a recent instance).
\begin{proposition}
Let $(G,\cL)$ be a twisted étale group bundle. Given a Radon probability measure $\mu$ on $G^0$, there exists a tracial state on $\cs_\ess(G,\mathcal L)$ that extends $\mu$.
\end{proposition}
\begin{proof}

Suppose first that $G$ can be covered by countably many bisections. Given $x \in G^0\setminus D^0$, the point mass $\delta_x$ is invariant, and Lemma \ref{point evaluation state} says that the canonical extension of $\delta_x$ to a tracial state on $\cs(G,\mathcal L)$ descends to a (tracial) state $\tau_x$ on $\cs_\ess(G,\mathcal L)$.

Since $G^0\setminus D^0$ is dense in $G^0$, the probability measure $\mu$ is in the closed convex hull of the point masses on $G^0\setminus D^0$. 
We conclude that there is a net $(\tau_i)$ of tracial states on $\cs_\ess(G,\mathcal L)$ such that, for every $f\in C_c(G^0)$, $\tau_i(f)\to \int f\,\mathrm d\mu$. By taking the limit of a convergent subnet of $(\tau_i)_i$ in the compact set of contractive traces, we obtain a trace on $\cs_\ess(G,\mathcal L)$ that extends $\mu$.

In the general case, write $G=\bigcup_k{H_k}$ as a directed union of open subgroupoids that contain $G^0$ and can be covered by countably many bisections. For each $k$, there is a state $\sigma_k$ on $\cs_{\ess}(G)$ such that $\sigma_k|_{\cs_{\ess}(H_k)}$ is a tracial state that extends $\mu$. By taking a limit point of $(\sigma_k)_k$, the conclusion follows.
\end{proof}

\section{Finite-state self-similar groups}\label{sec:ssg}
Let $X$ be a finite set such that $|X|\geq 2$. Let $X^*$ be the set of finite words on $X$, including the empty word, with the structure of rooted tree, in which two words are connected by an edge if and only if they are of the form $w$ and $wx$, for $w\in X^*$ and $x\in X$.

Let $\varphi\colon X^*\to X^*$ be an automorphism. Given $w\in X^*$, let $\varphi|_w\colon X^*\to X^*$ be the automorphism such that, for each $v\in X^*$, it holds that $\varphi(wv)=\varphi(w)\varphi|_w(v)$. We say that $\varphi$ is \emph{finite-state} if $\{\varphi|_w \mid w\in X^*\}$ is finite.

Note that $\varphi$ induces a homeomorphism (which we denote by $\varphi$ as well) $\varphi\colon X^\N\to X^\N$, such that $\varphi(w)_{[1,n]}=\varphi(w_{[1,n]})$, for any $w\in X^\N$ and $n\in \N$. We denote by $\mu$ the Bernoulli measure on $X^\N$.

The following result generalizes \cite[Theorem 4.2]{KSS06}.
\begin{theorem}\label{thm:gen}
Let $\varphi\colon X^*\to X^*$ be a finite-state automorphism. Then \[\mu(\Fix_\varphi\setminus\interior\Fix_\varphi)=0.\]
\end{theorem}
\begin{proof}
Given $k\geq 0$, let $P_k:=\{w\in X^k \mid \varphi(w)=w, \, \varphi|_w\neq e\}$ and $F_k:=\bigcup_{w\in P_k}wX^\N$. It is not difficult to see that $(F_k)_k$ is a decreasing sequence and that $\Fix_\varphi\setminus\interior\Fix_\varphi=\bigcap_k F_k$. Therefore, 
\begin{equation}\label{eq:fix}
\mu(\Fix_\varphi\setminus\interior\Fix_\varphi)=\lim_{k\to\infty}\frac{|P_k|}{|X|^k}.
\end{equation}

Since $\varphi$ is finite-state, there exists $p>0$ such that, for every $w\in X^*$ such that $\varphi|_w\neq e$, there exists $u\in X^p$ such that $\varphi|_w(u)\neq u$.

We claim that, for $k\geq 0$, $|P_{pk}|\leq (|X|^p-1)^k$. For $k=0$, this is obvious. Suppose it holds for $k-1$, with $k\geq 1$ fixed. 

Note that $P_{pk}\subset \{(w,u)\in P_{p(k-1)}\times X^p \mid \varphi|_w(u)=u\}$. Moreover, by the choice of $p$, given $w\in P_{p(k-1)}$, there exists $u\in X^p$ such that $\varphi|_w(u)\neq u$. Thus $|P_{pk}|\leq |P_{p(k-1)}|(|X^p|-1)\leq (|X|^p-1)^k$, by the induction hypothesis. This shows the claim.

Together with \eqref{eq:fix}, this concludes the proof.
\end{proof}

A \emph{faithful self-similar group action} $(\G,X)$ is a faithful action of a group $\G$ by automorphisms on the tree $X^*$ of a finite set $X$ such that $g|_v\in \G$ for every $g\in \G$ and $v \in X^*$. Then $(\G,X)$ is said to be \emph{finite-state} if each $g\in \G$ is finite-state. See the book \cite{NekraSSGbook} for a wealth of examples.

Let $(\G,X)$ be a faithful self-similar group. Given $u,v\in X^*$ such that $|u|=|v|$ and $g\in\G$, let $g_{u,v}\colon vX^\N\to uX^\N$ be given by $g_{u,v}(vw)=ug(w)$, for $w\in X^\N$. Let $K(\G,X)$ be the groupoid of germs associated with the pseudogroup $\{g_{u,v} \mid g\in \G, \, u,v\in X^*, \, |u|=|v|\}$. This groupoid was introduced in \cite[Section 5.3]{Nek09} as a model for the gauge-invariant algebra of $(\G,X)$. The Bernoulli measure $\mu$ on $X^\N$ is the only $K(\G,X)$-invariant probability measure. 

The following result follows easily from Proposition \ref{prop:pg} and Theorems \ref{thm:gen} and \ref{thm:u}. It was first shown by Yoshida (see \cite[Proposition 2.16 and Theorem 3.12]{Y19}) using a different technique for uniqueness of the tracial state, and assuming that the self-similar group is contracting, but his proof also works in this generality.

\begin{corollary}\label{cor:fssg}
Let $(\G,X)$ be a finite-state faithful self-similar group. Then the Bernoulli measure $\mu$ is essentially free with respect to $K(\G,X)$ and $\cs(K(\G,X))$ admits a unique tracial state.
\end{corollary} 
\begin{proof}
It suffices by Proposition \ref{prop:pg} to show that $\mu(\Fix_{g_{u,v}}\setminus\interior\Fix_{g_{u,v}})=0$ for any $u,v \in X^*$ with $|u| = |v|$ and $g \in \G$. If $u \ne v$, then $\Fix_{g_{u,v}} = \emptyset$. When $u = v$, we have $\mu(\Fix_{g_{u,u}}\setminus\interior\Fix_{g_{u,u}}) = |X|^{-|u|} \mu(\Fix_g\setminus\interior\Fix_g)$, which vanishes by Theorem \ref{thm:gen}. As $\mu$ is the unique invariant Radon probability measure on $X^{\mathbb N}$, its unique extension to a tracial state on $\cs(K(\G,X))$ guaranteed by Theorem \ref{thm:u} is the unique tracial state.
\end{proof}
For the essential $\cs$-algebra, we have the following consequence of Proposition \ref{prop:efe} and Corollary \ref{cor:fssg} (see also \cite[Theorem 3.12]{Y19}).
\begin{corollary}
Let $(\G,X)$ be a finite-state faithful self-similar group. Then the $\cs$-algebra $\cs_\ess(K(\G,X))$ admits a unique tracial state.
\end{corollary}

% \bibliography{bib}

\begin{bibdiv}
\begin{biblist}



\bib{BKM}{article}{
      author={Bardadyn, Krzysztof},
      author={Kwa{\'s}niewski, Bartosz},
      author={McKee, Andrew},
       title={Banach algebras associated to twisted \'etale groupoids: simplicity and pure infiniteness},
        date={2024},
     journal={arXiv:2406.05717},
}

\bib{BKKO}{article}{
      author={Breuillard, Emmanuel},
      author={Kalantar, Mehrdad},
      author={Kennedy, Matthew},
      author={Ozawa, Narutaka},
       title={{$\mathrm{C}^\ast$}-simplicity and the unique trace property for discrete groups},
        date={2017},
     journal={Publ. Math. Inst. Hautes \'Etudes Sci.},
      volume={126},
       pages={35\ndash 71},
}

% \bib{BM25}{article}{
%       author={Buss, Alcides},
%       author={Mart\'inez, Diego},
%        title={Essential groupoid amenability and nuclearity of groupoid {{\(\mathrm{C}^\ast\)}}-algebras},
%         date={2025},
%      journal={arXiv:2501.01775},
% }

\bib{BM26}{article}{
   author={Buss, Alcides},
   author={Mart\'inez, Diego},
   title={Essential groupoid amenability and nuclearity of groupoid {$\mathrm{C}^\ast$}-algebras},
   journal={J. Funct. Anal.},
   volume={290},
   date={2026},
   number={11},
   pages={Paper No. 111427, 54},
   issn={0022-1236},
%   review={\MR{5040091}},
%   doi={10.1016/j.jfa.2026.111427},
}

\bib{CEPSS}{article}{
 author = {Clark, Lisa Orloff},
 author = {Exel, Ruy},
 author = {Pardo, Enrique},
 author = {Sims, Aidan},
 author = {Starling, Charles},
 title = {Simplicity of algebras associated to non-{Hausdorff} groupoids},
 fjournal = {Transactions of the American Mathematical Society},
 journal = {Trans. Am. Math. Soc.},
 issn = {0002-9947},
 volume = {372},
 number = {5},
 pages = {3669--3712},
 year = {2019}
}

\bib{Coh13}{book}{
      author={Cohn, Donald~L.},
       title={Measure theory},
     edition={2nd revised ed.},
      series={Birkh{\"a}user Adv. Texts, Basler Lehrb{\"u}ch.},
   publisher={New York, NY: Birkh{\"a}user/Springer},
        date={2013},
        ISBN={978-1-4614-6955-1; 978-1-4614-6956-8},
}



\bib{Ex08}{article}{
      author={Exel, Ruy},
       title={Inverse semigroups and combinatorial {{\(\mathrm{C}^\ast\)}}-algebras},
        date={2008},
        ISSN={1678-7544},
     journal={Bull. Braz. Math. Soc. (N.S.)},
      volume={39},
      number={2},
       pages={191\ndash 313},
}

\bib{EP16}{article}{
 author = {Exel, Ruy},
 author = {Pardo, Enrique},
 title = {The tight groupoid of an inverse semigroup},
 fjournal = {Semigroup Forum},
 journal = {Semigroup Forum},
 issn = {0037-1912},
 volume = {92},
 number = {1},
 pages = {274--303},
 year = {2016}
}

\bib{ExelPitts22}{book}{
      author={Exel, Ruy},
      author={Pitts, David~R.},
       title={Characterizing groupoid {$\mathrm{C}^\ast$}-algebras of non-{H}ausdorff \'etale groupoids},
      series={Lecture Notes in Mathematics},
   publisher={Springer, Cham},
        date={2022},
      volume={2306},
        ISBN={978-3-031-05512-6; 978-3-031-05513-3},
         url={https://doi.org/10.1007/978-3-031-05513-3},
%      review={\MR{4510931}},
}

\bib{FellDoranI}{book}{
      author={Fell, J. M.~G.},
      author={Doran, R.~S.},
       title={Representations of {$^*$}-algebras, locally compact groups, and {B}anach {$^*$}-algebraic bundles. {V}ol. 1},
      series={Pure and Applied Mathematics},
   publisher={Academic Press, Inc., Boston, MA},
        date={1988},
      volume={125},
        ISBN={0-12-252721-6},
        note={Basic representation theory of groups and algebras},
%      review={\MR{936628}},
}


\bib{FollandRealAnalysis}{book}{
      author={Folland, Gerald~B.},
       title={Real analysis},
      series={Pure and Applied Mathematics (New York)},
   publisher={John Wiley \& Sons, Inc., New York},
        date={1984},
        ISBN={0-471-80958-6},
        note={Modern techniques and their applications, A Wiley-Interscience Publication},
%      review={\MR{767633}},
}


\bib{FollandAbstractHarmonicAnalysis}{book}{
      author={Folland, Gerald~B.},
       title={A course in abstract harmonic analysis},
     edition={Second},
      series={Textbooks in Mathematics},
   publisher={CRC Press, Boca Raton, FL},
        date={2016},
        ISBN={978-1-4987-2713-6},
%      review={\MR{3444405}},
}


\bib{GM26}{article}{
      author={Gabe, James},
      author={Miller, Alistair},
       title={Self-adjoint traces on the {Pedersen} ideal of {{\(\mathrm{C}^\ast\)}}-algebras},
        date={2026},
     journal={Publ. Mat.},
      volume={70},
      number={1},
       pages={277\ndash 301},
}

\bib{GGGKN24}{article}{
      author={Gardella, Eusebio},
      author={Geffen, Shirly},
      author={Gesing, Rafaela},
      author={Kopsacheilis, Grigoris},
      author={Naryshkin, Petr},
       title={Essential freeness, allostery and {$\mathcal Z$}-stability of crossed products},
        date={2024},
    journal={arXiv:2405.04343},
}

\bib{GLN23}{article}{
      author={Gong, Guihua},
      author={Lin, Huaxin},
      author={Niu, Zhuang},
       title={A review of the {Elliott} program of classification of simple amenable $\mathrm{C}^\ast$-algebras},
        date={2023},
     journal={arXiv:2311.14238},
}


\bib{Hum25}{article}{
      author={Hume, Jeremy~B.},
       title={Characterizations of zero singular ideal in \'etale groupoid {{\(\mathrm{C}^\ast\)}}-algebras via compressible maps},
        date={2025},
     journal={arXiv:2509.07262},
}


\bib{KS22a}{article}{
 author = {Kalantar, Mehrdad},
 author = {Scarparo, Eduardo},
 title = {Boundary maps, germs and quasi-regular representations},
 fjournal = {Advances in Mathematics},
 journal = {Adv. Math.},
 issn = {0001-8708},
 volume = {394},
 pages = {31},
 note = {Id/No 108130},
 year = {2022}
}


\bib{KS22}{article}{
      author={Kalantar, Mehrdad},
      author={Scarparo, Eduardo},
       title={Boundary maps and covariant representations},
        date={2022},
        ISSN={0024-6093},
     journal={Bull. Lond. Math. Soc.},
      volume={54},
      number={5},
       pages={1944\ndash 1961},
}

\bib{KSS06}{article}{
      author={Kambites, Mark},
      author={Silva, Pedro~V.},
      author={Steinberg, Benjamin},
       title={The spectra of lamplighter groups and {Cayley} machines.},
        date={2006},
        ISSN={0046-5755},
     journal={Geom. Dedicata},
      volume={120},
       pages={193\ndash 227},
}

\bib{KTT}{article}{
      author={Kawamura, Shinz{\^o}},
      author={Takemoto, Hideo},
      author={Tomiyama, Jun},
       title={State extensions in transformation group {$\mathrm{C}^\ast$}-algebras},
        date={1990},
     journal={Acta Sci. Math. (Szeged)},
      volume={54},
      number={1-2},
       pages={191\ndash 200},
}

\bib{Ker20}{article}{
      author={Kerr, David},
       title={Dimension, comparison, and almost finiteness},
        date={2020},
        ISSN={1435-9855,1435-9863},
     journal={J. Eur. Math. Soc. (JEMS)},
      volume={22},
      number={11},
       pages={3697\ndash 3745},
         url={https://doi.org/10.4171/jems/995},
%      review={\MR{4167017}},
}




\bib{Kum86}{article}{
      author={Kumjian, Alexander},
       title={On {$\mathrm{C}^\ast$}-diagonals},
        date={1986},
        ISSN={0008-414X,1496-4279},
     journal={Canad. J. Math.},
      volume={38},
      number={4},
       pages={969\ndash 1008},
         url={https://doi.org/10.4153/CJM-1986-048-0},
%      review={\MR{854149}},
}

\bib{KM21}{article}{
      author={{Kwa\'sniewski}, Bartosz~Kosma},
      author={{Meyer}, Ralf},
       title={{Essential crossed products for inverse semigroup actions: simplicity and pure infiniteness}},
        date={2021},
        ISSN={1431-0635},
     journal={{Doc. Math.}},
      volume={26},
       pages={271\ndash 335},
}

\bib{KMP}{article}{
      author={{Kwa\'sniewski}, Bartosz~Kosma},
      author={{Meyer}, Ralf},
      author={{Prasad}, Akshara},
       title={Type semigroups for twisted groupoids and a dichotomy for groupoid {{\(\mathrm{C}^\ast\)}}-algebras},
        date={2025},
     journal={arXiv:2502.17190},
}


% \bib{LZ24}{article}{
%       author={Li, Kang},
%       author={Zhang, Jiawen},
%        title={Tracial states on groupoid {{\(\mathrm{C}^\ast\)}}-algebras and essential freeness},
%         date={2024},
%      journal={J. Noncommut. Geom.},
%         note={Published online first},
% }

\bib{LZ26}{article}{
   author={Li, Kang},
   author={Zhang, Jiawen},
   title={Tracial states on groupoid {$\mathrm{C}^\ast$}-algebras and essential freeness},
   journal={J. Noncommut. Geom.},
   volume={20},
   date={2026},
   number={2},
   pages={561--586},
   issn={1661-6952},
%   review={\MR{5047641}},
%   doi={10.4171/jncg/597},
}


\bib{Li20}{article}{
      author={Li, Xin},
       title={Every classifiable simple {$\mathrm{C}^\ast$}-algebra has a {C}artan subalgebra},
        date={2020},
        ISSN={0020-9910,1432-1297},
     journal={Invent. Math.},
      volume={219},
      number={2},
       pages={653\ndash 699},
         url={https://doi.org/10.1007/s00222-019-00914-0},
%      review={\MR{4054809}},
}


% \bib{MS25}{article}{
%       author={Mart\'inez, Diego},
%       author={Szak\'acs, N{\'o}ra},
%        title={Algebraic singular functions are not always dense in the ideal of {$\mathrm{C}^\ast$}-singular functions},
%         date={2025},
%      journal={arXiv:2510.01947},
% }

\bib{MS26}{article}{
   author={Mart\'inez, Diego},
   author={Szak\'acs, N\'ora},
   title={Algebraic singular functions are not always dense in the ideal of
   {$\mathrm{C}^\ast$}-singular functions},
   journal={Bull. Lond. Math. Soc.},
   volume={58},
   date={2026},
   number={5},
   pages={Paper No. e70381, 23},
   issn={0024-6093},
%   review={\MR{5077044}},
%   doi={10.1112/blms.70381},
}

\bib{Mat12}{article}{
      author={Matui, Hiroki},
       title={Homology and topological full groups of \'etale groupoids on totally disconnected spaces},
        date={2012},
     journal={Proc. Lond. Math. Soc. (3)},
      volume={104},
      number={1},
       pages={27\ndash 56},
}



\bib{MS24}{article}{
      author={Miller, Alistair},
      author={Steinberg, Benjamin},
       title={Homology and {K}-theory for self-similar actions of groups and groupoids},
        date={2024},
     journal={arXiv:2409.02359},
}

 \bib{MM03}{book}{
 author = {Moerdijk, Ieke},
 author = {Mr{\v{c}}un, J.},
 title = {Introduction to foliations and {Lie} groupoids},
 fseries = {Cambridge Studies in Advanced Mathematics},
 series = {Camb. Stud. Adv. Math.},
 volume = {91},
 isbn = {0-521-83197-0},
 year = {2003},
 publisher = {Cambridge: Cambridge University Press}
 }

\bib{M90}{book}{
      author={Murphy, Gerard~J.},
       title={$\mathrm{C}^\ast$-algebras and operator theory},
   publisher={Boston, MA etc.: Academic Press, Inc.},
        date={1990},
        ISBN={0-12-511360-9},
}

\bib{NekraSSGbook}{book}{
      author={Nekrashevych, Volodymyr},
       title={Self-similar groups},
      series={Mathematical Surveys and Monographs},
   publisher={American Mathematical Society, Providence, RI},
        date={2005},
      volume={117},
        ISBN={0-8218-3831-8},
         url={https://doi.org/10.1090/surv/117},
%      review={\MR{2162164}},
}

\bib{Nek09}{article}{
      author={Nekrashevych, Volodymyr},
       title={{$\mathrm{C}^*$}-algebras and self-similar groups},
        date={2009},
     journal={J. Reine Angew. Math.},
      volume={630},
       pages={59\ndash 123},
}

\bib{NS22}{article}{
      author={Neshveyev, Sergey},
      author={Stammeier, Nicolai},
       title={The groupoid approach to equilibrium states on right {LCM} semigroup {{\(\mathrm{C}^\ast\)}}-algebras},
        date={2022},
        ISSN={0024-6107},
     journal={J. Lond. Math. Soc., II. Ser.},
      volume={105},
      number={1},
       pages={220\ndash 250},
}


\bib{NO19}{article}{
      author={{Nyland}, Petter},
      author={{Ortega}, Eduard},
       title={{Topological full groups of ample groupoids with applications to graph algebras}},
        date={2019},
        ISSN={0129-167X},
     journal={{Int. J. Math.}},
      volume={30},
      number={4},
       pages={66},
        note={Id/No 1950018},
}


\bib{Ped18}{book}{
      author={Pedersen, Gert~K.},
       title={{$\mathrm{C}^\ast$}-algebras and their automorphism groups},
     edition={Second},
      series={Pure and Applied Mathematics (Amsterdam)},
   publisher={Academic Press, London},
        date={2018},
        note={Edited and with a preface by S\o ren Eilers and Dorte Olesen},
}

\bib{Ren08}{article}{
      author={Renault, Jean},
       title={Cartan subalgebras in {$\mathrm{C}^\ast$}-algebras},
        date={2008},
        ISSN={0791-5578},
     journal={Irish Math. Soc. Bull.},
      number={61},
       pages={29\ndash 63},
%      review={\MR{2460017}},
}

\bib{SSW20}{book}{
      author={Sims, Aidan},
      author={Szab\'o, G\'abor},
      author={Williams, Dana},
       title={Operator algebras and dynamics: groupoids, crossed products, and {R}okhlin dimension},
      series={Advanced Courses in Mathematics. CRM Barcelona},
   publisher={Birkh\"auser/Springer, Cham},
        date={2020},
}

\bib{Tay25}{article}{
      author={Taylor, Jonathan},
       title={Essential {C}artan subalgebras of {$\mathrm{C}^\ast$}-algebras},
        date={2025},
     journal={Math. Scand.},
      volume={131},
      number={3},
       pages={559\ndash 604},
}

\bib{Urs21}{article}{
      author={Ursu, Dan},
       title={Characterizing traces on crossed products of noncommutative {{\(\mathrm{C}^\ast\)}}-algebras},
        date={2021},
        ISSN={0001-8708},
     journal={Adv. Math.},
      volume={391},
       pages={29},
        note={Id/No 107955},
}

\bib{Win18}{inproceedings}{
      author={Winter, Wilhelm},
       title={Structure of nuclear {$\mathrm{C}^\ast$}-algebras: from quasidiagonality to classification and back again},
        date={2018},
   booktitle={Proceedings of the {I}nternational {C}ongress of {M}athematicians---{R}io de {J}aneiro 2018. {V}ol. {III}. {I}nvited lectures},
   publisher={World Sci. Publ., Hackensack, NJ},
       pages={1801\ndash 1823},
%      review={\MR{3966830}},
}

\bib{Y19}{article}{
      author={Yoshida, Keisuke},
       title={A von {Neumann} algebraic approach to self-similar group actions},
        date={2019},
        ISSN={0129-167X},
     journal={Int. J. Math.},
      volume={30},
      number={14},
       pages={19},
        note={Id/No 1950074},
}

\end{biblist}
\end{bibdiv}
% \bibliographystyle{alpha}


\end{document}